\documentclass[onefignum,onetabnum]{article}

\usepackage[margin=1in]{geometry}
\usepackage{lipsum}
\usepackage{amsfonts,amsthm,amsmath,amssymb}
\usepackage{graphicx}
\usepackage{epstopdf}
\usepackage{combelow}
\usepackage{algorithmic}
\usepackage{cleveref}
\ifpdf
  \DeclareGraphicsExtensions{.eps,.pdf,.png,.jpg}
\else
  \DeclareGraphicsExtensions{.eps}
\fi

\newtheorem{definition}{Definition}
\newtheorem{proposition}{Proposition}

\title{Reduced operator inference for nonlinear partial differential equations}

\author{Elizabeth Qian\footnote{Department of Computing + Mathematical Sciences, California Institute of Technology, Pasadena, CA
  (eqian@caltech.edu, \url{http://www.elizabethqian.com}).}
  \and Ionu\cb{t}-Gabriel Farca\cb{s}\footnote{Oden Institute for Computational Engineering and Sciences, University of Texas at Austin, Austin, TX
  (ionut.farcas@austin.utexas.edu, kwillcox@oden.utexas.edu, \url{https://kiwi.oden.utexas.edu}).}
\and Karen Willcox\footnotemark[2]
}

\usepackage{amsopn}

\newcommand{\R}{\mathbb{R}}
\newcommand{\pdv}[3][]{\frac{\partial^{#1}#2}{\partial#3^{#1}}}
\newcommand{\dv}[3][]{\frac{\text{d}^{#1}#2}{\text{d}#3^{#1}}}

\newcommand{\final}{\text{final}}
\newcommand{\init}{\text{init}}

\newcommand{\cspod}{s_{\rm pod}}    
\newcommand{\csopi}{s_{\rm opinf}}    
\newcommand{\csp}{s_{\rm proj}}     
\newcommand{\dcsp}{\dot s_{\rm proj}} 

\newcommand{\bT}{\mathbf{T}}

\newcommand{\bw}{\mathbf{w}}

\newcommand{\cC}{\mathcal{C}}
\newcommand{\cT}{\mathcal{T}}
\newcommand{\cX}{\mathcal{X}} 
\newcommand{\cY}{\mathcal{Y}} 
\newcommand{\cJ}{\mathcal{J}} 

\usepackage{graphicx}
\usepackage{wrapfig}
\usepackage{lscape}
\usepackage{rotating}
\usepackage{epstopdf}
\usepackage{xcolor}

\definecolor{red}{HTML}{000000}

\usepackage[numbers]{natbib}
\usepackage{bibentry}

\newcommand{\ignore}[1]{}
\newcommand{\nobibentry}[1]{{\let\nocite\ignore\bibentry{#1}}}

\begin{document}

\maketitle

\begin{abstract}
	We present a new scientific machine learning method that learns from data a computationally inexpensive surrogate model for predicting the evolution of a system governed by a time-dependent nonlinear partial differential equation (PDE), an enabling technology for many computational algorithms used in engineering settings.
	Our formulation generalizes to the function space PDE setting the Operator Inference method previously developed in [B. Peherstorfer and K. Willcox, \textit{Data-driven operator inference for non-intrusive projection-based model reduction}, Computer Methods in Applied Mechanics and Engineering, 306 (2016)] for systems governed by ordinary differential equations.
	The method brings together two main elements. First, ideas from projection-based model reduction are used to explicitly parametrize the learned model by low-dimensional polynomial operators which reflect the known form of the governing PDE. Second, supervised machine learning tools are used to infer from data the reduced operators of this physics-informed parametrization. For systems whose governing PDEs contain more general (non-polynomial) nonlinearities, the learned model performance can be improved through the use of \emph{lifting} variable transformations, which expose polynomial structure in the PDE. 
	The proposed method is demonstrated on two examples: a heat equation model problem that demonstrates the benefits of the function space formulation in terms of consistency with the underlying continuous truth, and a
  three-dimensional combustion simulation with over 18 million degrees of freedom, for which the learned reduced models achieve accurate predictions with a dimension reduction of five orders of magnitude and model runtime reduction of {\color{red} up to nine} orders of magnitude.
\end{abstract}

\section{Introduction}
Systems governed by nonlinear PDEs are ubiquitous in {\color{red} engineering and scientific} application and traditional numerical solvers based on high-dimensional spatial discretizations are {\color{red}computationally} expensive {\color{red}even on powerful supercomputers}. The development of efficient {\color{red}reduced} models for PDEs is therefore an enabling technology for many-query applications such as optimization or uncertainty quantification. 
Empirical successes in using machine learning to learn complex nonlinear functions from data have motivated many works which use machine learning tools to learn models for scientific problems governed by PDEs. The model learning task in such settings is to learn a map that describes the state evolution within the infinite-dimensional function space in which the spatially continuous PDE state lies. 
In this work, we consider specifically the task of learning a map that describes the state evolution within a low-dimensional subspace of the underlying infinite dimensional function space, in an approach that combines elements from supervised learning and model reduction.
Our method, \emph{Operator Inference for PDEs}, generalizes to the Hilbert space setting {\color{red}the} method~\cite{Peherstorfer16DataDriven} previously developed in Euclidean space which learns a {\color{red}reduced} model that reflects the structure of the system governing equations by drawing on ideas from projection-based model reduction.

In projection-based model reduction, a low-dimensional reduced model for the evolution of a nonlinear PDE is derived by projecting the PDE operators onto a low-dimensional subspace of the underlying infinite-dimensional Hilbert space. 
Several approaches to defining this approximation subspace have been proposed (see~\cite{willcox2015MORreview} for a review), but a common empirical choice is the proper orthogonal decomposition (POD) subspace, which spans the leading principal components of an available set of state data~\cite{berkooz1993proper,lumley1967structure,sirovich87turbulence}.
For PDEs with only linear and polynomial terms, the projection-based reduced model admits an efficient low-dimensional numerical representation that can be cheaply evaluated at a cost independent of the dimension of the original model~\cite{benner2015two,willcox2015MORreview,cuong2005certified,goyal2016algebraic,hesthaven2016certified}. For PDEs with more general nonlinear terms, an additional level of approximation, e.g., interpolation~\cite{astrid2008missing,barrault2004empirical,carlberg2013gnat,chaturantabut2010,drmac2016new,grepl2007efficient,nguyen2008best}, is generally needed to achieve computational efficiency. Alternatively, the works in~\cite{benner2015two,goyal2016algebraic,kramer2018lifting,KW2019_balanced_truncation_lifted_QB} transform PDEs with general nonlinear terms to representations with only quadratic nonlinearities via the use of structure-exposing \emph{lifting} transformations, and then project the quadratic operators of the resultant lifted PDE to obtain a reduced model.

Traditionally, projection-based model reduction is \emph{intrusive}, because numerical implementation of the reduced models requires access to codes that implement the PDE operators. This is a limitation in settings where the code is inaccessible (e.g., legacy or commercial codes). 
Methods for \emph{non-intrusive} reduced modeling use data to learn reduced operators: one approach is to learn a map from an input (such as a parameter or initial condition) to the coefficients of the PDE state in a reduced basis. This mapping can be parametrized by radial basis functions~\cite{Nair2013}, a neural network~\cite{mainini2017data,swischuk2018physics,wang2019non} or a nearest-neighbors regression~\cite{swischuk2018physics}. However, these works do not incorporate knowledge of the system governing equations into the learning parametrizations. 
The Operator Inference method proposed in~\cite{Peherstorfer16DataDriven} draws on projection-based model reduction to use knowledge of a system's governing ODEs to explicitly learn the projection-based reduced operators for ODE systems containing low-order polynomial nonlinear terms. 
The incorporation of knowledge of the governing equations allows Operator Inference to provably recover the true intrusive reduced polynomial operators of the governing ODEs under certain conditions~\cite{Peherstorfer16DataDriven,Peherstorfer19Reprojection}. 
The work in~\cite{QKPW19LiftLearn} presents the \emph{Lift \& Learn} method, which applies Operator Inference method to more general nonlinear ODEs by applying quadratic lifting transformations to ODE state data and then fitting reduced linear and quadratic operators to the lifted data. In a different approach, the work in~\cite{loiseau2018constrained} learns sparse reduced models using sparse regression.

While the aforementioned methods for model reduction are frequently applied to scientific models for systems governed by PDEs, the methods are most often developed for the ODE setting corresponding to models arising from spatial discretization of the PDEs. In this work, we {\color{red}formulate the Operator Inference learning approach} in the spatially continuous function space setting for the first time. The function space formulation has the following advantages over the Euclidean formulation: (i) The function space formulation is mesh-independent and consistent with the underlying infinite-dimensional system: e.g., it yields implementations that use inner products that approximate the underlying infinite-dimensional inner product, whereas the ODE formulation using the Euclidean inner product may yield approximations that are inconsistent with the underlying infinite-dimensional system. (ii) The function space formulation allows the {\color{red}Lift \& Learn approach to take advantage of more general (in particular, nonlocal) structure-exposing lifting transformations between function spaces, compared to previous work in the ODE setting~\cite{QKPW19LiftLearn} that considered only pointwise transformations.} (iii) The function space setting enables rigorous analysis of error with respect to the underlying function space truth: for example, the reduced-basis community has an extensive literature on function space error estimates for intrusive reduced models~\cite{grepl2005posteriori,haasdonk2011efficient,rozza2008reduced,veroy2002posteriori,veroy2005certified}. 

We emphasize that Operator Inference is not intended to replace intrusive model reduction approaches when intrusive approaches are feasible: rather, Operator Inference provides a path to obtaining a reduced model when intrusive approaches such as POD-Galerkin are not feasible. The new function space formulation of Operator Inference that we present here makes clear that non-intrusive reduced modeling approaches can and should be viewed in the context of recent works in supervised learning between function spaces.
For example, the work in~\cite{nelsen2021random} formulates a random feature model for approximating input-output maps on function spaces. Another approach~\cite{lu2021learning} simultaneously trains two neural networks, one which learns basis functions in which to approximate a solution function, and another which learns the coefficients of the learned basis functions. However, these architectures treat the data as arising from a black box, and do not incorporate structure in system governing equations into their learning architectures. The works in~\cite{anandkumar2020neural,li2020fourier} use knowledge of kernel operator structure to inspire a novel neural network architecture for approximating maps between function spaces for PDEs. {\color{red}In this work}, we use knowledge of the system governing equations to explicitly parametrize the PDE evolution map by low-dimensional operators which reflect the structure of the PDE, inspired by work in projection-based model reduction.

Our emphasis on learning polynomial operators for possibly transformed system states bears similarities to approaches that seek to learn the linear Koopman operator for observables of nonlinear dynamical systems. Koopman operator theory states that every nonlinear dynamical system can be exactly described by an infinite-dimensional linear operator that acts on scalar observables of the system~\cite{koopman1932dynamical,mezic2005spectral,mezic2013analysis,kutz2016dynamic}. DMD learns a finite-dimensional approximation of the Koopman operator in a given set of observables using a least-squares minimization~\cite{williams2015data,schmid2010dynamic}. The success of this approach is dependent on whether the chosen set of observables defines a space that is invariant or close to invariant with respect to the Koopman operator. Previous works have used dictionary-learning~\cite{li2017extended}, neural networks~\cite{takeishi2017learning}, and kernel methods~\cite{kevrekidis2016kernel} to define the space of observables. {\color{red} Another recent approach uses neural networks to define coordinates in which terms of a governing equation are identified from a large dictionary of possible terms using sparse regression~\cite{champion2019data}.} In contrast, we propose a physics-informed approach which chooses learning variables based on the known governing PDE.

The remainder of the paper is organized as follows. \Cref{sec: setting} formulates the prediction problem for a system governed by a nonlinear PDE and provides background on POD-Galerkin model reduction. \Cref{sec: operator inference PDE} presents our Operator Inference formulation for learning mappings between Hilbert spaces which approximate the dynamics of a nonlinear PDE. \Cref{sec: lifting} introduces Lift \& Learn in the PDE setting, which uses variable transformations to apply Operator Inference to learn quadratic models for non-quadratic PDEs. \Cref{sec: numerical experiments} presents numerical results for two examples: (i) a heat equation example which demonstrates that the ODE formulation can lead to results that are inconsistent with the underlying truth, and (ii) a three-dimensional rocket combustion simulation with over 18 million degrees of freedom that takes 45,000 CPU hours to run, which demonstrates the {\color{red} potential of the model learning approach to yield computational gains for high-dimensional engineering simulations}. Conclusions are discussed in \Cref{sec: conclusions}.

\section{Setting and technical foundations}\label{sec: setting}
This section introduces the initial\slash boundary-value problem for which we seek {\color{red}a reduced} model (\Cref{ssec: problem formulation}) and presents the POD-Galerkin approach to model reduction (\Cref{ssec: POD background}) on which our method builds.

\subsection{Problem formulation}\label{ssec: problem formulation}
Let $\Omega$ be a bounded physical domain in $\R^d$ with Lipschitz continuous boundary $\partial\Omega$, and let $[0,T_\final]\subset\R$ be a time {\color{red}interval} of interest. Let $s:\Omega\times[0,T_\final]\to\R^{d}$
denote the $d$-dimensional physical state. We denote by $s_j:\Omega\times[0,T_\final]\to\R$ the $j$-th element of the state $s$, so that
\begin{align}
  s(x,t) = \begin{pmatrix}
    s_1(x,t) \\ s_2(x,t) \\ \vdots \\ s_d(x,t)
  \end{pmatrix}.
\end{align} 
For $j = 1,2,\ldots,d$, let $\cX_j$ be a separable Hilbert space of real-valued functions of $\Omega$. 
We denote by $\cX= \cX_1 \times \cX_2 \times \cdots \times \cX_{d}$ the product of the spaces $\cX_j$
and endow $\cX$ with the natural definition of inner product:
\begin{align}
  \big<s,s'\big>_\cX = \sum_{j=1}^{d}\left<s_j,s'_j\right>_{\cX_j}.
\end{align}
Let $f: \cX\to \cX$ denote a smooth nonlinear {\color{red}mapping between function spaces}.
We then consider the (strong form of the) following partial differential equation, written as an ordinary differential equation in the Hilbert space $\cX$:
\begin{align}
  \label{eq: nonlinear IBVP}
  \pdv{s}t &= f(s), \qquad s({\color{red}\cdot,} \, 0) = s_\init.
\end{align}
We remark that in our numerical examples in \Cref{sec: numerical experiments}, the $\cX$-inner product is the $L^2(\Omega)$-inner product; however, our formulation is not restricted to this choice. In what follows, we therefore use $\left<\cdot,\cdot,\right>_\cX$ to denote an abstract choice of inner product for the Hilbert space $\cX$.

\subsection{POD-Galerkin model reduction}\label{ssec: POD background}
The POD-Galerkin approach to model reduction defines a reduced model for an initial\slash boundary value problem by projecting the operators of the governing PDE onto a low-dimensional subspace spanned by the leading modes of a proper orthogonal decomposition of an available set of state data.
Let $s^{(1)},\,s^{(2)},\,\ldots,\,s^{(K)}\in \cX$ denote a set of $K$ snapshots of the solution to \cref{eq: nonlinear IBVP}, 
taken at different times (and more generally, possibly from multiple trajectories corresponding to e.g., varying initial conditions). The POD subspace of dimension $r$, denoted $\cX_r$, is the rank-$r$ subspace of $\cX$ that minimizes the mean square error between the snapshots and their projections onto $\cX_r$. Let $\cC: \cX\to \cX$ denote the empirical covariance operator of the snapshots, defined as
\begin{align}
  \cC \psi = \sum_{k=1}^K \big<\psi, s^{(k)}\big>_\cX \; s^{(k)}, \quad \psi\in \cX.
\end{align}
Because $\cC$ is compact, non-negative, and self-adjoint on $\cX$, there exists a sequence of eigenvalues $\lambda_i$ and an associated complete orthonormal basis $\psi_i \in \cX$ satisfying $\cC\psi_i = \lambda_i\psi_i$, with $\lambda_i\to0 \text{ as }i\to\infty$.
The leading $r$ eigenfunctions $\{\psi_i\}_{i=1}^r$ are a POD basis of rank $r$. That is, $\{\psi_i\}_{i=1}^r$ satisfy $\left<\psi_i,\psi_j\right>_\cX=\delta_{ij}$, and are a solution to
\begin{align}
  \min_{\{\psi_i\}_{i=1}^r\in \cX} \sum_{k=1}^K\left \| s^{(k)} - \sum_{l=1}^r \big<s^{(k)}, \psi_l\big>_\cX\psi_l\right\|_\cX^2,
\end{align}
with attained minimum sum-of-squares objective value $\sum_{i = r+1}^{K} \lambda_i$.

The POD-Galerkin approximation to \cref{eq: nonlinear IBVP} approximates the state, $s(x,t)$, in the span of the POD basis as follows:
\begin{align}
  \cspod(x,t) = \sum_{i=1}^r \hat s_i(t) \psi_i(x),
   \label{eq: POD approx continuous}
\end{align}
and evolves the POD state $\cspod(x,t)$ within the POD subspace $\cX_r$ by enforcing Galerkin orthogonality of the PDE residual to each function in the POD basis:
\begin{align}
  \left<\pdv{\cspod}t,\psi_l\right>_\cX = \big<f(\cspod),\psi_l\big>_\cX, \qquad l = 1,2,\ldots,r.
  \label{eq: Galerkin condition continuous}
\end{align}
Substituting \cref{eq: POD approx continuous} into \cref{eq: Galerkin condition continuous} yields the following system of $r$ ordinary differential equations (ODEs) for the evolution of $\hat s_1(t),\hat s_2(t),\ldots,\hat s_r(t)$, the coefficients of the POD basis functions:
\begin{subequations}
  \label{eq: POD galerkin IBVP nonlin continuous}
  \begin{align}
    \dv{\hat s_l}t &= \left<f\left(\sum_{i=1}^r \hat s_i(t) \psi_i(x)\right),\psi_l(x)\right>_\cX,\label{eq: POD dynamics nonlin continuous}\\
    \hat s_l(0) &= \big<s_\init(x),\psi_l(x)\big>_\cX,
  \end{align}
\end{subequations}
for $l =1,2,\ldots,r$. The initial-value problem in~\cref{eq: POD galerkin IBVP nonlin continuous} is referred to as the {\color{red}\textit{reduced model}}. However, evaluating the $r$-dimensional projected dynamics in~\cref{eq: POD dynamics nonlin continuous} requires evaluation of the map $f$ in the infinite-dimensional Hilbert space $\cX$ before projecting back down to $\cX_r$. 

In the special case where $f$ contains only polynomial nonlinearities, the ODEs which govern the evolution of the POD coefficients \cref{eq: POD dynamics nonlin continuous} have structure that allows rapid online solution of \cref{eq: POD galerkin IBVP nonlin continuous}.
For example, suppose $f(s) = a(s) + h(s,s)$,
where $a: \cX\to \cX$ is linear and $h: \cX\times \cX\to \cX$ bilinear. We note that this definition allows $h(s,s)$ to represent quadratic terms such as $s^2$ or $s\pdv sx$.  Then, \cref{eq: POD dynamics nonlin continuous} takes the form
\begin{align}
    \dv{\hat s_l}t & = \sum_{i=1}^r \hat s_i(t) \big<a \big(\psi_i(x)\big),\psi_l(x)\big>_\cX +\sum_{i,j=1}^r \hat s_i(t)\hat s_j(t) \big< h\big(\psi_i(x),\psi_j(x)\big),\psi_l(x)\big>_\cX.
  \label{eq: POD polynomial continuous}
\end{align}
If we define the reduced state $\hat s(t)\in\R^r$ by $\hat s(t) = \begin{pmatrix}
    \hat s_1(t) & \hat s_2(t) & \cdots & \hat s_r(t)
  \end{pmatrix}^\top$,
then \cref{eq: POD polynomial continuous} can be equivalently expressed as
\begin{align}
  \dv{\hat s}t = \hat A \hat s + \hat H (\hat s \otimes\hat s),
\end{align}
where $\otimes$ denotes the Kronecker product, and $\hat A\in\R^{r\times r}$ and $\hat H\in\R^{r\times r^2}$ are reduced linear and quadratic operators. Denote by $\hat a_{l,i}$ the $i$-th entry of the $l$-th row of $\hat A$ and denote by $\hat h_{l,ij}$ the $((i-1)r+j)$-th entry of the $l$-th row of $\hat H$. Then, the entries of $\hat A$ and $\hat H$ are given by $\hat a_{l,i} = \left<\psi_l, a(\psi_i)\right>_\cX$, and $ \hat h_{l,ij} = \left<\psi_l, h(\psi_i,\psi_j)\right>_\cX$, for $l,i,j=1,2,\ldots,r$.
Computing the reduced operators $\hat A$ and $\hat H$ is considered \emph{intrusive} because it requires access to the PDE operators $a$ and $h$. We propose a \emph{non-intrusive} method for learning the reduced operators from data in \Cref{sec: operator inference PDE}. 

\section{Operator Inference for PDEs}\label{sec: operator inference PDE}
This section presents our formulation of the Operator Inference method for learning reduced models for nonlinear PDEs. \Cref{ssec: least squares formulation} introduces the learning formulation as a linear least-squares problem. \Cref{ssec: numerical considerations} discusses numerical properties of the method. 

\subsection{{\color{red}Formulation}}\label{ssec: least squares formulation}
Operator Inference \emph{non-intrusively} learns a reduced model for \cref{eq: nonlinear IBVP} from state snapshot and time derivative data. That is, suppose that for each state snapshot $s^{(k)}$, for $k = 1,2,\ldots,K$, the corresponding time derivative $\dot s^{(k)} = f(s^{(k)})$ is available.
We will approximate the PDE state in $\cX_r$, the POD subspace of rank $r$:
\begin{align}
  \csopi(x,t) = \sum_{i=1}^r \tilde s_i(t) \psi_i(x),
\end{align}
where the basis functions $\psi_i$ are defined as in \Cref{ssec: POD background}.
Then, informed by the form of the POD-Galerkin reduced model, we fit to the available data a model for the evolution of $\csopi$ within $\cX_r$ with the following form:
\begin{align}
  \pdv{\csopi}t = f_r(\csopi(x,t);\tilde a_{l,i},\tilde h_{l,ij}) 
\end{align}
where $f_r(\cdot\,; \tilde a_{l,i}, \tilde h_{l,ij}) : \cX_r\to \cX_r$ has the explicit polynomial form
\begin{align}
  f_r(\csopi(x,t); \tilde a_{l,i}, \tilde h_{l,ij}) = \sum_{l=1}^r \left(\sum_{i=1}^r\tilde a_{l,i} \tilde s_i(t) + \sum_{i,j=1}^r\tilde h_{l,ij} \tilde s_i(t) \tilde s_j(t)\right) \psi_l(x),
  \label{eq: explicit poly op inf form}
\end{align}
where $\tilde s_i(t) = \left<\psi_i(x),\csopi(x,t)\right>$.
To learn the parameters $\tilde a_{l,i}$ and $\tilde h_{l,ij}$, the snapshot and time derivative data are projected onto the POD subspace $\cX_r$ as follows:
\begin{align}
  \csp^{(k)} = \sum_{l=1}^r\big< s^{(k)},\psi_l\big>_\cX\psi_l,
  &&
  \dcsp^{(k)} = \sum_{l=1}^r \big<  \dot s^{(k)},\psi_l\big>_\cX \psi_l. 
\end{align}
The parameters $\tilde a_{l,i}$ and $\tilde h_{l,ij}$, for $l,i,j=1,2,\ldots,r$, are fit by minimizing the following least-squares objective:
\begin{align}
  \min_{\tilde a_{l,i}\in\R, \tilde h_{l,ij}\in\R} \frac1K\sum_{k=1}^K \left\| f_r\big(\csp^{(k)};\,\tilde  a_{l,i},\tilde h_{l,ij}\big) - \dcsp^{(k)}\right\|_\cX^2.
  \label{eq: continuous least squares}
\end{align}
Since $\cX_r$ is isomorphic to $\R^r$, we can collect the POD coefficients of the state and time derivative data as follows:
\begin{align}
  \tilde s^{(k)} = \begin{pmatrix}
    \left<s^{(k)},\psi_1\right>_\cX \\  \left<s^{(k)},\psi_2\right>_\cX  \\ \vdots \\  \left<s^{(k)},\psi_r\right>_\cX 
  \end{pmatrix}, 
  \quad
  \dot{\tilde s}^{(k)} = \begin{pmatrix}
    \left<\dot s^{(k)},\psi_1\right>_\cX  \\ \left<\dot s^{(k)},\psi_2\right>_\cX \\ \vdots \\ \left<\dot s^{(k)},\psi_r\right>_\cX
  \end{pmatrix}, 
  \label{eq: reduced data vector}
\end{align}
for $k=1,2,\ldots,K$. Then, \cref{eq: continuous least squares} can be equivalently expressed as
\begin{align}
  \arg\min_{\tilde A\in\R^{r\times r}, \tilde H\in\R^{r\times r^2}} \frac1K \sum_{k=1}^K \left\| \tilde A \tilde s^{(k)} + \tilde H \big(\tilde s^{(k)} \otimes \tilde s^{(k)} \big) - \dot{\tilde s}^{(k)}\right\|_{\R^r}^2,
  \label{eq: continuous least-squares iso}
\end{align}
where $\tilde a_{l,i}$ is the $i$-th entry of the $l$-th row of $\tilde A$ and $\tilde h_{l,ij}$ is the $((i-1)r+j)$-th entry of the $l$-th row of $\tilde H$. The inferred matrix operators $\tilde A$ and $\tilde H$ define a reduced model of the form
\begin{align}
  \pdv{\tilde s}t = \tilde A \tilde s + \tilde H (\tilde s \otimes \tilde s).
\end{align}

Due to the inherent symmetry in the quadratic terms $\tilde h_{l,ij}\tilde s_i(t)\tilde s_j(t)$ and $\tilde h_{l,ji}\tilde s_j(t) \tilde s_i(t)$ in \cref{eq: explicit poly op inf form} the solution to \cref{eq: continuous least squares} (and therefore to \cref{eq: continuous least-squares iso}) is not unique. We therefore specify that we seek the solution to \cref{eq: continuous least squares} (equivalently, to \cref{eq: continuous least-squares iso}) that minimizes the sum of the operator Frobenius norms:
\begin{align}
  \left\|\tilde A \right\|_F^2 + \left\|\tilde H \right\|_F^2 = \sum_{l,i,j=1}^r \left| \tilde h_{l,ij} \right|^2 + \sum_{l,i=1}^r \left|\tilde a_{l,i} \right|^2.
  \label{eq: operator norm}
\end{align}
The solution to \cref{eq: continuous least squares}\slash\cref{eq: continuous least-squares iso} that minimizes \cref{eq: operator norm} will yield a \emph{symmetric} quadratic operator $\tilde H$ in the sense that $\tilde h_{l,ij} = \tilde h_{l,ji}$.

\subsection{Numerical considerations}\label{ssec: numerical considerations}
The formulation of Operator Inference as a linear least-squares problem has several numerical advantages. First, note that the least-squares problem has the following matrix algebraic formulation:
\begin{align}
   D \begin{bmatrix}
     \tilde A^\top \\ \tilde H^\top
   \end{bmatrix} = \dot{\tilde S}^\top,
   \label{eq: matrix algebra}
\end{align} 
where the least-squares data matrix $D\in\R^{r\times (r+r^2)}$  and right-hand side $\dot{\tilde S}\in\R^{r\times K}$ are given by
\begin{align}
  D = \begin{bmatrix}
    (\tilde s^{(1)})^\top & \big(\tilde s^{(1)}\otimes \tilde s^{(1)}\big)^\top  \\
    (\tilde s^{(2)})^\top & \big(\tilde s^{(2)}\otimes \tilde s^{(2)}\big)^\top  \\
    \vdots & \vdots \\
    (\tilde s^{(K)})^\top & \big(\tilde s^{(K)}\otimes \tilde s^{(K)}\big)^\top  \\
  \end{bmatrix},
  \qquad
  \dot{\tilde S}^\top = \begin{bmatrix}
    (\dot{\tilde s}^{(1)})^\top \\ 
    (\dot{\tilde s}^{(2)})^\top\\
    \vdots \\
    (\dot{\tilde s}^{(K)})^\top
  \end{bmatrix}.
\end{align}
Note that \cref{eq: matrix algebra} in fact defines $r$ independent least-squares problems, one for each row of the inferred operators $\tilde A$ and $\tilde H$. The rows of the inferred operators are independently fit to the $r$ columns of $\dot{\tilde S}^\top$, containing data for the evolution of each of the $r$ components of the reduced state $\tilde s$.
\Cref{eq: matrix algebra} can be scalably solved for all $r$ right-hand sides by standard linear algebra routines, and additionally, standard QR methods inherently find the norm-minimizing least-squares solution which defines a symmetric quadratic operator $\tilde H$.

Second, since the norm-minimizing solution to \cref{eq: matrix algebra} satisfies $\tilde h_{l,ij} = \tilde h_{l,ji}$, each of the $r$ independent least-squares problems defined by \cref{eq: matrix algebra} has in fact only $r + {r\choose 2} = r + \frac{r(r+1)}2$ degrees of freedom. Thus, the number $K$ of state and time derivative pairs in the data set need only satisfy $K> r+ {r\choose 2}$ for \cref{eq: matrix algebra} to admit a unique solution with symmetric $\tilde H$ (assuming linear independence of the $K$ data pairs). Since $r$ is typically chosen to be small in the context of projection-based model reduction, this data requirement is small relative to the requirements of many black box machine learning methods which assume the availability of `big data'. Limiting the data requirements of the method is also our motivation for focusing on the inference of quadratic reduced operators: while the method can in principle be extended to infer matrix operators corresponding to higher-order polynomial terms, the number of degrees of freedom in the least-squares problem grows exponentially as the polynomial order is increased.

Finally, the sensitivity of the linear least-squares problem to perturbations in the data is well-understood and can be ameliorated through standard regularization methods, 
 for example, by adding weights to the Frobenius norm penalty described in~\cref{eq: operator norm}, as in
in the Frobenius norm of the reduced operators:
\begin{align}
  \begin{split}
    \arg\min_{\substack{\tilde A\in\R^{r\times r},\\ \tilde H\in\R^{r\times r^2}}} \bigg(\frac1K&\sum_{k=1}^K \left\| \tilde A \tilde s^{(k)}  + \tilde H \big(\tilde s^{(k)} \otimes \tilde s^{(k)} \big) - \dot{\tilde s}^{(k)}\right\|_{\R^r}^2 +\gamma_1 \left\|\tilde A\right\|_F^2 + \gamma_2 \left\|\tilde H\right\|_F^2\bigg).
  \end{split}
  \label{eq: LL regularized LS}
\end{align}
This is equivalent to adding a weighted Euclidean norm penalty for each of the $r$ independent least-squares problems whose solutions define rows of $\tilde A$ and $\tilde H$, and yields a symmetric operator $\tilde H$ as described earlier. However, by increasing the regularization weights, the coefficients in the operators can be driven closer to zero.
Explicit regularization also has the effect of making the Operator Inference problem well-posed in the case where even our modest data requirement cannot be met. Operator Inference is therefore especially amenable to the task of learning surrogate models for the high-dimensional simulations that typically arise in scientific settings, where data are often limited due to the expense of the high-dimensional simulations.

\section{Lift \& Learn for PDEs with non-quadratic nonlinearities}\label{sec: lifting}
The Operator Inference approach of \Cref{sec: operator inference PDE}, which learns linear and quadratic reduced operators, can be made applicable to PDEs with more general nonlinear terms through the use of lifting variable transformations which expose quadratic structure in the PDE. This section formulates in the PDE setting the \emph{Lift \& Learn} method introduced in~\cite{QKPW19LiftLearn} for ODEs.
\Cref{ssec: lifting motivation} motivates our consideration of quadratic structure-exposing transformations. Lifting is defined in \Cref{ssec: lifting}, and the Lift \& Learn approach is presented in \Cref{ssec: lift and learn summary}.

\subsection{Motivation}\label{ssec: lifting motivation}
While the Operator Inference method fits linear and quadratic reduced operators to data, the method itself makes no assumption of linear or quadratic structure in the governing PDE. Operator Inference can, in principle, be applied to fit a quadratic reduced model to systems governed by PDEs with arbitrary nonlinearities. However, when the Operator Inference parametrization reflects the form of the governing PDE, the Operator Inference model will have the same form as the POD-Galerkin reduced model. This fact is used in~\cite{QKPW19LiftLearn} to bound in the ODE setting the mean square error of the reduced model over the data by the error due to projection onto the POD subspace.
Additionally, the work in~\cite{Peherstorfer16DataDriven} uses this fact to show in the ODE setting that Operator Inference non-intrusively recovers the intrusive POD-Galerkin reduced model operators asymptotically, while the work in~\cite{Peherstorfer19Reprojection} proves pre-asymptotic recovery guarantees in the ODE setting under certain conditions. Similar analysis applies to the PDE setting considered here~\cite{QianPhDthesis}. These theoretical analyses motivate our consideration of lifting maps which expose quadratic structure in general nonlinear PDEs.

\subsection{Lifting maps}\label{ssec: lifting}
We seek to expose quadratic structure in a general nonlinear PDE through the use of \emph{lifting} maps which transform and/or augment the PDE state with auxiliary variables. 

\subsubsection{Definition of quadratic lifting}
Let $\cY=\cY_1\times \cY_2\times \cdots\times \cY_{d'}$ denote a separable Hilbert space with $d'\geq d$, and inner product $\left<\cdot,\cdot\right>_\cY$. Let $\cT: \cX \to \cY$ denote a continuous and differentiable nonlinear map. Then, let $w(\cdot,t) = \cT(s(\cdot,t))$, and let $\cJ$ denote the Jacobian of $\cT$.
\begin{definition}
  \label{def: lifting map}
  Consider the nonlinear PDE given by
  \begin{align}
    \pdv s t = f(s).
    \label{eq: repeated nonlin PDE}
  \end{align}
  The tuple $(\cT, a, h)$ is called a \emph{quadratic lifting} of \cref{eq: repeated nonlin PDE},
  and the $d'$-dimensional field $w(x,t)$ is the called the lifted state, if $w$ satisfies
  \begin{align}
    \pdv{w}t = \pdv{\cT(s)}t = \cJ(s)\pdv{s}t = \cJ(s)f(s) = a(\cT(s)) + h(\cT(s)) = a(w) + h(w).
    \label{eq: lifted dynamics definition}      
  \end{align}
  The map $\cT$ is called the \emph{lifting map}.
\end{definition}

\subsubsection{Lifted initial\slash boundary-value problem}
We can use \Cref{def: lifting map} to reformulate the original nonlinear initial/boundary value problem \cref{eq: nonlinear IBVP}.  We seek a lifted solution $w:[0,T_\final]\to \cY$ satisfying:
\begin{align}
  \pdv wt &= a(w) + h(w), \qquad w(0) = \cT(s_\init).
  \label{eq: lifted IBVP}
\end{align}

\begin{proposition}
  \label{prop: equivalence PDE}
  Suppose the tuple $(\cT, a,h)$ is a quadratic lifting as defined in \Cref{def: lifting map}. Then, if $s(x,t)$ solves \cref{eq: nonlinear IBVP}, $w(x,t) = \cT(s(x,t))$ is a solution of \cref{eq: lifted IBVP}.
\end{proposition}

\begin{proof}
  It is trivial to verify that if $s$ satisfies the original initial condition, then $\cT(s)$ satisfies the lifted initial condition. Then, note that because $s$ satisfies \cref{eq: nonlinear IBVP}, $\pdv{}t\cT(s) = \cJ(s)f(s).$
  By \Cref{def: lifting map}, $\cJ(s)f(s) = a(\cT(s)) + h(\cT(s))$, so $w$ satisfies \cref{eq: lifted IBVP}.
\end{proof}

\subsubsection{An example quadratic lifting}\label{sssec: lifting example}
  Consider the nonlinear reaction-diffusion equation with cubic reaction term:
  \begin{align}
    \label{eq: cubic reaction diffusion}
    \pdv{s}t = f(s) = \pdv[2]s x - s^3.
  \end{align}
  Let $\cT$ be defined as
  \begin{align}
    \cT: s \mapsto \begin{pmatrix}
      s \\ s^2
    \end{pmatrix} \equiv \begin{pmatrix}
      w_1 \\ w_2
    \end{pmatrix} = w.
    \label{eq: reaction diffusion cubic lifting}
  \end{align}
  Then, the evolution of $w$ satisfies $\pdv w t = a(w) + h(w,w)$,
  where $a$ and $h$ are defined as
  \begin{align}
    a(w) = \begin{pmatrix}
      \pdv[2]{w_1}x \\ 0
    \end{pmatrix},
    \qquad
    h(w,w') = \begin{pmatrix}
      -w_1 w'_2 \\ 2w_1 \pdv[2]{w'_1}x - 2 w_2 w'_2
    \end{pmatrix}.
    \label{eq: reaction diffusion forms}
  \end{align}

\begin{proposition}
  The tuple $(\cT,a,h)$ given by \cref{eq: reaction diffusion cubic lifting,eq: reaction diffusion forms} is a quadratic lifting of \cref{eq: cubic reaction diffusion}. 
\end{proposition}

\begin{proof}
  Note that the Jacobian of $\cT$ is given by $\cJ(s) = \begin{pmatrix}
      1 \\ 2s 
    \end{pmatrix}$.
  Thus,
  \begin{align}
  	\begin{split}
  		\pdv w t &= \cJ(s) f(s) = \begin{pmatrix}
	      1 \\ 2s 
	    \end{pmatrix} \left(\pdv[2]s x - s^3\right)
	    = \begin{pmatrix}
	      \pdv[2]s x - s^3 \\ 
	      2s\pdv[2]s x - 2s^4
	    \end{pmatrix} \\
	    &= \begin{pmatrix}
	      \pdv[2]{w_1}x - w_1w_2 \\
	      2w_1\pdv[2]{w_1}x - 2(w_2)^2
	    \end{pmatrix} 
	    =  a(w) + h(w,w),
  	\end{split}
  \end{align}
  for $a,h$ defined in \cref{eq: reaction diffusion forms}, so the tuple $(\cT,a,h)$ satisfies \Cref{def: lifting map}.
\end{proof}

\subsubsection{Discussion}
Lifting maps allow many systems with general nonlinearities to be reformulated with only polynomial nonlinearities, and systems with higher-order polynomial nonlinearities to be reformulated with only quadratic nonlinearities~\cite{gu2011qlmor}. While there are no universal guarantees of the existence of a lifting, liftings have been derived for many nonlinear PDEs. For example, the lifting of the cubic reaction-diffusion PDE in \Cref{sssec: lifting example} can be extended to lift reaction-diffusion PDEs with higher-order polynomial source terms to quadratic form with additional auxiliary variables~\cite{QianPhDthesis}. For fluids problems, both the compressible Euler equations and the Navier-Stokes equations admit a quadratic formulation based on the specific volume representation of the fluid state variables~\cite{balajewicz2016minimal,QKPW19LiftLearn}. Quadratic liftings have also been derived for the FitzHugh-Nagumo neuron activation model~\cite{kramer2018lifting} and the Chafee-Infante equation~\cite{benner2015two}. 
In the numerical experiments in \Cref{sec: numerical experiments}, we will consider a \emph{partial} lifting (in which most---but not all---terms of the lifted PDEs are quadratic in the lifted state) of the PDEs governing a three-dimensional rocket combustion simulation.
While we wish to emphasize that the non-intrusive nature of the Operator Inference method allows the flexibility to learn models for PDEs with arbitrary non-quadratic nonlinearities, the identification of even a partial quadratic lifting of the governing PDE can improve the ability of the inferred quadratic operators to model the system dynamics in the lifted variables.

\subsection{Lift \& Learn for PDEs}\label{ssec: lift and learn summary}
Once a lifting map has been identified for a nonlinear PDE, the Lift \& Learn approach takes state snapshot $\{s_k\}_{k=1}^K$ and time derivative data $\{\dot s_k\}_{k=1}^K$ in the original nonlinear variable representation and applies the lifting map to each snapshot and time derivative:
\begin{align}
  w_k= \cT(s_k), \qquad \dot w_k = \cJ(s_k) \dot s_k, \qquad k = 1,2,\ldots,K.
  \label{eq: lifting snapshots}
\end{align}
The lifted snapshots $\{w_k\}_{k=1}^K$ are used to define a POD basis of rank $r$ in the lifted variables. We then apply the Operator Inference approach of \Cref{sec: operator inference PDE} to the lifted data to obtain a lifted reduced model. 

We emphasize that it is the non-intrusive nature of the Operator Inference method that enables the practical development of reduced models in lifted variables. While our knowledge of the lifted governing equations enables us to write down expressions for the projection-based reduced lifted operators, computing these reduced operators in practice requires numerical evaluation of the $a(\cdot)$ and $h(\cdot,\cdot)$ forms, which are generally not available: even in settings where code for the original governing equations is accessible, code for the lifted operators is generally not available. The Operator Inference approach allows us to learn reduced lifted operators in such practical settings.

\section{Numerical examples}\label{sec: numerical experiments}
We now present numerical experiments from two examples: the first example, presented in \Cref{ssec: heat equation}, is a one-dimensional heat equation, which demonstrates {\color{red}that choosing an inner product for implementation that is inconsistent with the underlying truth leads to POD basis functions inconsistent with the underlying function space as well as learned operators that are sub-optimal with respect to the underlying function space norm.} 
The second example is a three-dimensional simulation of a model rocket combustor that demonstrates the {\color{red} potential of our learning approach to scale to complex high-dimensional engineering problems.}
\Cref{ssec: nonlin CVRC eqs} presents the nonlinear governing equations of the combustion model, and \Cref{ssec: lifted CVRC eqs} presents the transformation of the governing equations to approximately quadratic form. In \Cref{ssec: CVRC test problem}, we introduce the specific test problem we consider and the learning task. \Cref{ssec: CVRC op inf formulation} details the reduced model learning formulation. Finally, \Cref{ssec: CVRC prediction results} presents and discusses the prediction performance of our learned models.

\subsection{{\color{red}One-dimensional heat} equation}\label{ssec: heat equation}
We consider the heat equation,
\begin{align}
	\pdv ut = \pdv[2]ux,
\end{align}
on the unit spatial domain $x\in(0,1)$ with homogeneous Dirichlet boundary conditions and random Gaussian initial conditions as described below. 
The heat equation is discretized using a first-order explicit scheme in time and a second-order central difference scheme in space. The spatial discretization uses a non-uniform mesh: the grid spacing in $(0,0.3)$ is $\Delta x_1 = 0.05$, the spacing in $(0.3,0.7)$ is $\Delta x_2 = 0.01$, and the spacing in $(0.7,1)$ is $\Delta x_3=0.1$. This non-uniform spatial discretization was chosen to illustrate the effect of using different norms (Euclidean vs. $L^2(\Omega)$) in the Operator Inference formulation. Initial conditions $u_0(x)$ are drawn randomly as follows: let $L>1$ and let $\{\xi_l\}_{l=1}^L$ be independently and identically distributed according to a unit normal distribution. Then, a single random initial condition is given by:
\begin{align}\label{eq: heat rand init}
 	u_0(x) = \sqrt{2}\sum_{l=1}^L \xi_l(l\pi)^{-\frac32}\sin(l\pi x).
\end{align} 
Initial conditions drawn randomly according to the above correspond to random draws from the Gaussian random field $\mathcal{N}(0,C_L)$, where $C_L$ is the projection of the covariance kernel $C=\left(\pdv[2]{}x\right)^{-\frac32}$ onto its $L$ leading principal components. The true principal components of this covariance kernel are the sin functions $\sqrt{2}\sin(l\pi x)$, for $l = 1,2,\ldots,L$. Example initial conditions on our non-uniform mesh are shown in \Cref{fig: heat GRF}.

\begin{figure}[h]
	\centering
	\includegraphics[width=0.5\textwidth]{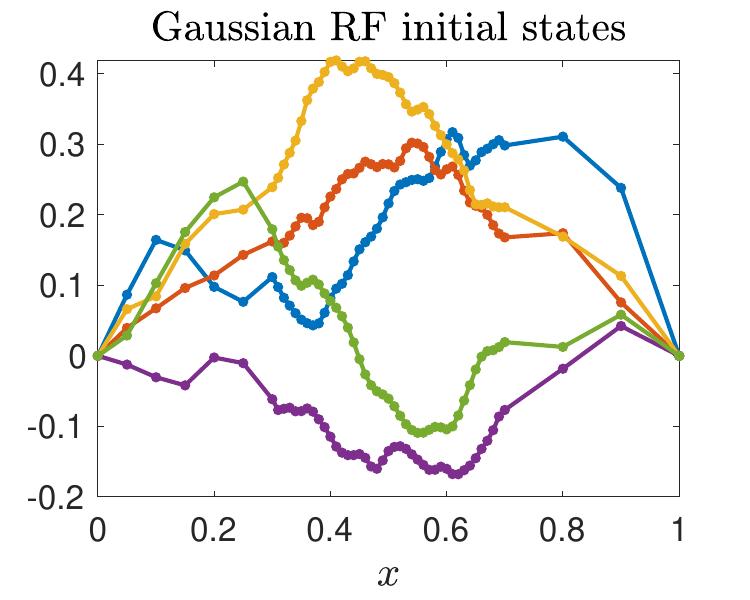}
	\caption{Example initial conditions drawn randomly according to~\cref{eq: heat rand init}. Line markers show non-uniform grid spacing.}\label{fig: heat GRF}
\end{figure}

We draw 1{\color{red},}000 initial conditions randomly as described with $L=60$ and evolve the discretized heat equation from $t = 0$ to $t = 0.01$ with time step $\Delta t = 10^{-4}$. The discretized state is saved at every tenth time step. Trajectories from the first 500 initial conditions are used as training data for Operator Inference and the latter 500 trajectories are reserved for testing. We compare results from the function space formulation presented in this work with results from the ODE formulation in earlier works. For the function space formulation, we use a discrete inner product that approximates the $L^2([0,1])$ inner product using the {\color{red}trapezoidal} rule, whereas the ODE formulation uses the Euclidean inner product.

\begin{figure}[h]
	\centering
	\includegraphics[width=0.32\textwidth]{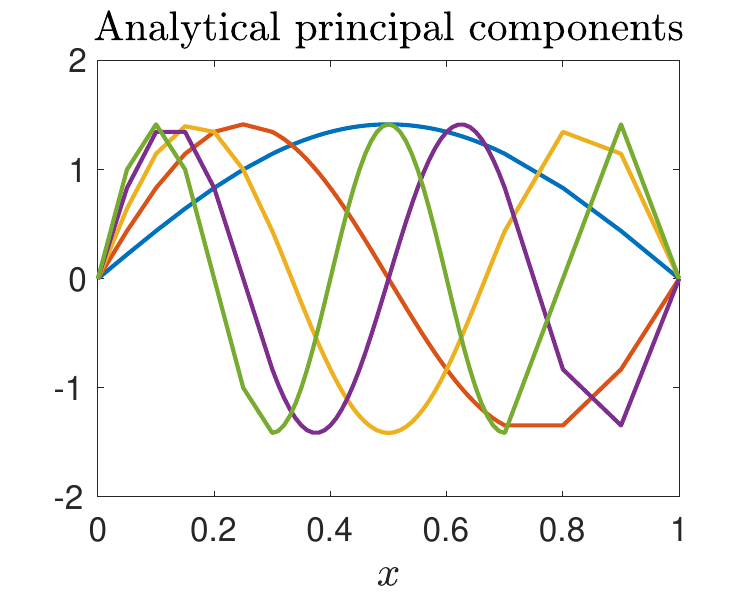}
	\includegraphics[width=0.32\textwidth]{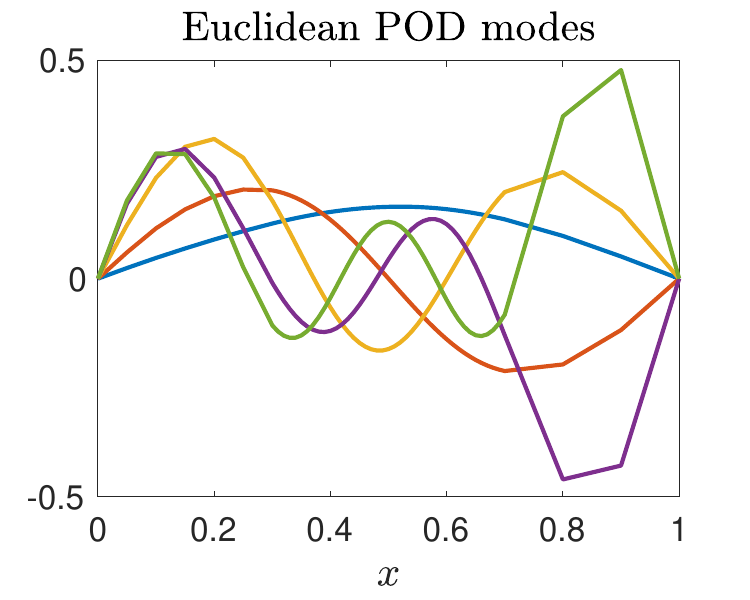}
	\includegraphics[width=0.32\textwidth]{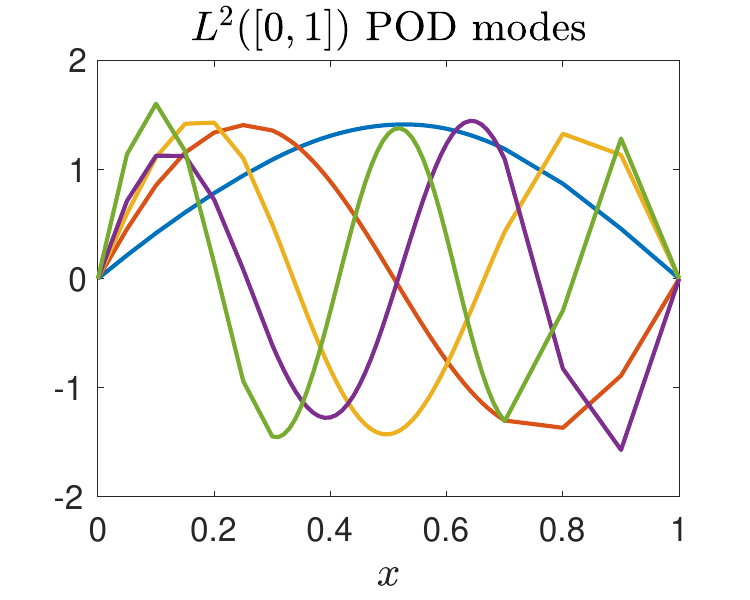}
	\caption{Leading POD basis functions of the heat equation data. Left: analytical principal components. Middle: empirical POD modes in the Euclidean inner product. Right: empirical POD modes in the $L^2([0,1])$ inner product.}\label{fig: heat POD modes}
\end{figure}

\Cref{fig: heat POD modes} compares the empirical POD modes of the data using the $L^2([0,1])$-inner product and the Euclidean inner product with the analytical principal components of the underlying covariance kernel. We note that the use of the $L^2([0,1])$-inner product in the function space formulation leads to empirical POD modes that are consistent with the underlying infinite-dimensional truth, whereas the Euclidean inner product used in the ODE setting leads to empirical POD modes that are inconsistent with the underlying distribution.

\begin{figure}[h]
	\centering
	\includegraphics[width=0.45\textwidth]{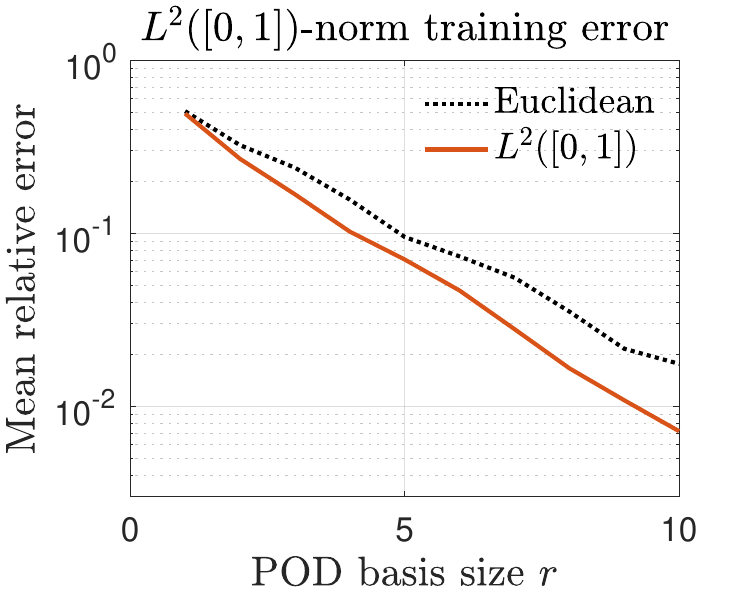}\hfill
	\includegraphics[width=0.45\textwidth]{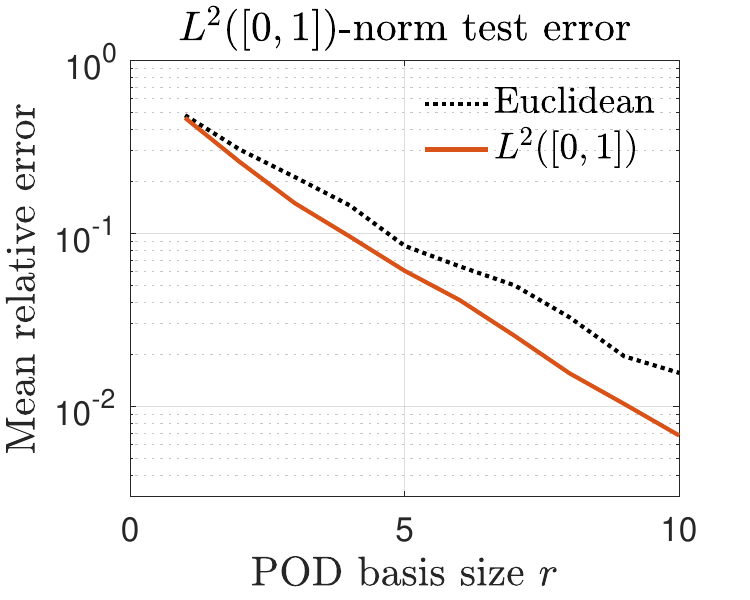}
	\caption{Training and test errors for Operator Inference learned models using function space formulation and Euclidean formulation.}\label{fig: heat errors}
\end{figure}

We use both the function space Operator Inference method and the Euclidean space Operator Inference method to learn reduced models of sizes $r = 1,2,\ldots, 10$ (the original size of the non-uniform mesh is $n=50$). These learned reduced models are then used to integrate from the reduced initial conditions in time. \Cref{fig: heat errors} compares the mean relative $L^2([0,1])$-norm errors of the predictions from each learned model. For this heat equation example, there is a clear improvement in prediction accuracy with respect to the underlying function space norm when the function space norm is used in Operator Inference.

Code for this example is available at {\tt https://github.com/elizqian/operator-\\inference}, which also contains examples of Operator Inference applied to other problems. Demonstration code for the Lift \& Learn approach can be found at {\tt https://github.com/elizqian/transform-and-learn}.

\subsection{Combustion problem: original governing PDEs}\label{ssec: nonlin CVRC eqs}
We now demonstrate the Operator Inference approach {\color{red}on a large-scale problem by} learning a reduced model for a three-dimensional simulation of the Continuously Variable Resonance Combustor (CVRC), an experiment at Purdue University that has been extensively studied experimentally and computationally. {\color{red}In particular, several recent works consider learning surrogate models for quasi-one dimensional~\cite{wang2019non,qian2019transform} and two-dimensional~\cite{swischuk2020learning,mcquarrie2021data} versions of its governing equations.} 
The governing PDEs for the CVRC test problem {\color{red} that we consider} are the three-dimensional compressible Navier-Stokes equations in conservative form with a flamelet/progress variable chemical model~\cite{ihme2008prediction,pierce2001progress}. In this chemical model, the chemistry is parametrized by two variables, the mixture mean, denoted $Z_m$, and the reaction progress variable, denoted $C$. Thus, there are $d = 7$ nonlinear state variables:
\begin{align}
  s = \begin{pmatrix}
    \rho & \rho u_1 & \rho u_2 & \rho u_3 &  \rho e & \rho Z_m & \rho C
  \end{pmatrix}^\top,
  \label{eq: CVRC conservative variables}
\end{align}
where $\rho$ is the density, $u_1$, $u_2$, and $u_3$ are the velocities in the three spatial directions, and $e$ is the specific energy. Together, the variables $Z_m$ and $C$ specify a \emph{flamelet manifold} for the chemical species concentrations and reaction source term based on looking up values from a pre-computed table. The governing PDE is given by
\begin{align}
  \pdv st + \nabla\cdot \big(F(s)- F_v(s)\big) = g(s),\label{eq: CVRC cons equations}
\end{align}
where $\nabla = \begin{pmatrix}
  \pdv{}{x_1} & \pdv{}{x_2} & \pdv{}{x_3}
\end{pmatrix}^\top$, $F(s)$ is the inviscid flux, given by
\begin{align}
  F(s) = \begin{pmatrix}
    \rho u_1 \\ \rho u_1^2 + p \\ \rho u_1u_2 \\ \rho u_1 u_3 \\ \rho u_1 h^0 \\ \rho u_1 Z_m \\ \rho u_1 C
  \end{pmatrix}\hat x_1
  +
  \begin{pmatrix}
    \rho u_2 \\ \rho u_1 u_2 \\ \rho u_2^2 + p \\ \rho u_2 u_3 \\ \rho u_2 h^0 \\ \rho u_2 Z_m \\ \rho u_2 C
  \end{pmatrix}\hat x_2
  +
  \begin{pmatrix}
    \rho u_3 \\ \rho u_1 u_3 \\ \rho u_2 u_3 \\ \rho u_3^2 + p \\ \rho u_3 h^0 \\ \rho u_3 Z_m \\ \rho u_3 C
  \end{pmatrix}\hat x_3,
  \label{eq: inviscid flux}
\end{align}
$F_v(s)$ is the viscous flux, given by
\begin{align}
  \begin{split}
    F_v(s) &= 
  \begin{pmatrix}
    0 \\ \tau_{11} \\ \tau_{21} \\ \tau_{31} \\ \beta_1 \\ \rho D \pdv{Z_m}{x_1} \\ \rho D \pdv C{x_1}
  \end{pmatrix}\hat x_1
  +
  \begin{pmatrix}
    0 \\ \tau_{12} \\ \tau_{22} \\ \tau_{32} \\ \beta_2 \\ \rho D \pdv{Z_m}{x_2} \\ \rho D \pdv C{x_2}
  \end{pmatrix}\hat x_2
  +
  \begin{pmatrix}
    0 \\ \tau_{13} \\ \tau_{23} \\ \tau_{33} \\ \beta_3 \\ \rho D \pdv{Z_m}{x_3} \\ \rho D \pdv C{x_3}
  \end{pmatrix}\hat x_3,
  \end{split}
  \label{eq: viscous flux}
\end{align}
where $\beta_i = \sum_{j=1}^3 u_j \tau_{ji} + \kappa \pdv{T}{x_i} -\rho\sum_{l=1}^{n_{\rm sp}}D\pdv{Y_l}{x_i} h_l$, for $i = 1,2,3$, 
and $g$ is the chemical  source term, given by $ g(s) = \begin{pmatrix}
    0 & 0 & 0 & 0 & 0 & 0 & \dot\omega_C 
  \end{pmatrix}^\top$.
The total enthalpy $h^0$ satisfies $\rho e = \rho h^0 -p$. The viscous shear stresses are given by $\tau_{ij} = \mu \left(\pdv{u_i}{x_i} + \pdv{u_j}{x_j} - \frac23 \pdv{u_m}{x_m}\delta_{ij}\right)$,
where the three directions are summed over in the term indexed by $m$. 

The temperature $T$ is computed under the ideal gas assumption, $T = \frac p{R\rho}$, where $R = \frac{\mathcal{R}}{W_\text{mol}}$ is the gas constant, $\mathcal{R}=8.314\frac{\text{kJ}}{\text{K}\cdot\text{kmol}}$ is the universal gas constant, and $W_\text{mol} = \left(\sum_{l=1}^{n_{\rm sp}} \frac{Y_l}{W_l}\right)$ is the molecular weight of the species mixture,
where $Y_l$ are species mass fractions of the $n_{\rm sp}$ individual chemical species and $W_l$ are their molecular weights.
The mass diffusivity $D$ is assumed to be equal to thermal diffusivity $\alpha$ under the unit Lewis number assumption,
$ D = \alpha = \frac{\kappa}{\rho c_p}, $
where $\kappa$ is the thermal conductivity, and $c_p = \pdv hT$ is the specific heat capacity, defined to be the derivative of the enthalpy $h$ with respect to the temperature $T$. The enthalpy $h = \sum_{l} h_l Y_l$ is a convex combination of the species enthalpies, weighted by the species mass fractions, where $h_l$ denotes the enthalpy of the $l$-th chemical species.
The enthalpy of the $l$-th species has the following dependence on temperature:
\begin{align}
  \label{eq: enthalpy species}
  \frac{h_l(T)}{\mathcal{R}/W_l} = \begin{cases}
    \begin{split}
      -\frac{a_{l,1,1}}T + a_{l,1,2}\ln T + a_{l,1,3}T + \frac{a_{l,1,4}}2 T^2 \\ + \frac{a_{l,1,5}}3 T^3 + \frac{a_{l,1,6}}4 T^4 + \frac{a_{l,1,7}}5 T^5 + a_{l,1,8}
    \end{split}
    ,\quad 200 {\rm K} \leq T \leq 1000 {\rm K}, \\ 
    \\
    \begin{split}
      -\frac{a_{l,2,1}}T + a_{l,2,2}\ln T + a_{l,2,3}T + \frac{a_{l,2,4}}2 T^2 \\+ \frac{a_{l,2,5}}3 T^3 + \frac{a_{l,2,6}}4 T^4 + \frac{a_{l,2,7}}5 T^5 + a_{l,2,8}
    \end{split}
    ,\quad 1000 {\rm K} \leq T \leq 6000 {\rm K},
  \end{cases}
\end{align}
where the coefficients $a_{l,i,j}$ are given for each species $l$.
The species mass fractions are a function of the flamelet variables, $Y_l=Y_l(Z_m,C)$, where the exact relationship is defined by interpolating between values in a pre-computed table. The source term for the progress variable equation is also defined by a pre-computed table based on the flamelet variables, $\dot\omega_C = \dot\omega_C(Z_m,C)$. 

{\color{red}Previous work on reduced modeling for a different, two-dimensional model of the CVRC has shown that reduced models in the conservative state variables can suffer from a lack of robustness and stability~\cite{huang2019investigations}. In order to apply the Lift \& Learn formulation of~\Cref{sec: lifting} to this problem, we seek a partially lifted state variable representation, which we describe in the next section.}

\subsection{Combustion problem: lifted governing PDEs}\label{ssec: lifted CVRC eqs}
Due to the many nonlinearities in the governing equations that are non-quadratic in the conservative state variables, we seek a variable transformation that exposes quadratic structure in these nonlinearities in order to learn quadratic reduced operators. For this complex problem, we use a \emph{partial} lifting, i.e., a variable transformation that transforms many---but not all---of the nonlinearities in governing PDEs into quadratic nonlinearities.
Inspired by the quadratic representation of the compressible Euler equations in the specific volume variables~\cite{QKPW19LiftLearn}, we transform the flow variables in the CVRC governing equations from their conservative representation to their specific volume representation, and retain the flamelet variables in their conservative form. Additionally, we add the fluid temperature $T$ to the lifted state because it is a key quantity of interest and we wish to model it directly without recourse to look-up tables. Thus, the lifted variable representation with $d'=8$ is given by
\begin{align}
  w = \begin{pmatrix}
    \zeta & u_1 & u_2 & u_3 & p & \rho Z_m & \rho C & T
  \end{pmatrix}^\top,
  \label{eq: CVRC lifted variables}
\end{align}
where $\zeta = \frac1\rho$.
When the specific heat capacity $c_p$ is a constant, the governing equations in the lifted variables \cref{eq: CVRC lifted variables} contain mostly quadratic nonlinear terms. These governing equations are derived in \Cref{appendix} and are given by:
\begin{subequations}
\label{eq: lifted CVRC equations}
\begin{align}
  \pdv\zeta t & = -\nabla \zeta \cdot u + \zeta(\nabla \cdot u) \label{eq: CVRC zeta}, \\
  \pdv {u_1} t & = - \zeta \pdv p{x_1} - u \cdot \nabla u_1  + \zeta \left(\pdv{\tau_{11}}{x_1} + \pdv{\tau_{12}}{x_2} + \pdv{\tau_{13}}{x_3}\right), \label{eq: CVRC ux}\\
  \pdv {u_2} t & =  - \zeta \pdv p{x_2} - u \cdot \nabla u_2  + \zeta \left(\pdv{\tau_{21}}{x_1} + \pdv{\tau_{22}}{x_2} + \pdv{\tau_{23}}{x_3}\right), \label{eq: CVRC uy}\\
  \pdv {u_3} t & =  - \zeta \pdv p{x_3} - u \cdot \nabla u_3 + \zeta \left(\pdv{\tau_{31}}{x_1} + \pdv{\tau_{32}}{x_2} + \pdv{\tau_{33}}{x_3}\right), \label{eq: CVRC uz}\\
  \pdv pt &= - \gamma p(\nabla\cdot u) - (u\cdot\nabla p) - u_1\left(\pdv{\tau_{11}}{x_1} + \pdv{\tau_{12}}{x_2} + \pdv{\tau_{13}}{x_3}\right)\cdots\label{eq: pressure evol CVRC}\\
    &\qquad - u_2\left(\pdv{\tau_{12}}{x_1} + \pdv{\tau_{22}}{x_2} + \pdv{\tau_{23}}{x_3}\right) - u_3\left(\pdv{\tau_{13}}{x_1} + \pdv{\tau_{23}}{x_2} + \pdv{\tau_{33}}{x_3}\right)\cdots\nonumber\\ 
    & \qquad+ \pdv{}{x_1}(u_1\tau_{11} + u_2\tau_{12} + u_3\tau_{13}) + \pdv{}{x_2}(u_1\tau_{21} + u_2\tau_{22} + u_3\tau_{23}) \cdots\nonumber\\
    &\qquad+ \pdv{}{x_3}(u_1\tau_{31} + u_2\tau_{32} + u_3\tau_{33}) + \pdv{q_1}{x_1} + \pdv{q_2}{x_2} + \pdv{q_3}{x_3},\nonumber
   \\
  \pdv{\rho Z_m}t &=  -\pdv{u_1 (\rho  Z_m)}{x_1} -\pdv{ u_2 (\rho Z_m)}{x_2} - \pdv{u_3(\rho  Z_m)}{x_3} +\frac{\kappa}{c_p}\nabla^2 Z_m \label{eq: CVRC rho Z} \\
  \pdv{\rho C}t & = -\pdv{u_1(\rho C)}{x_1} - \pdv{u_2 (\rho C)}{x_2}- \pdv{u_3(\rho  C)}{x_3} + \frac\kappa{c_p}\nabla^2 C + \dot\omega_C\label{eq: dCdt CVRC}\\
  \pdv Tt &= p\pdv \zeta t + \zeta\pdv pt.\label{eq: CVRC dTdt}
\end{align}
\end{subequations}
Note that \cref{eq: CVRC zeta,eq: CVRC ux,eq: CVRC uy,eq: CVRC uz} contain only quadratic dependencies in the lifted variables since the stresses $\tau_{ij}$ are linear in the velocity components. In \cref{eq: pressure evol CVRC}, only the last three terms contain non-quadratic linearities, due to the dependence of the heat fluxes on the unmodeled species mass fractions: $q_i = -\frac\kappa R \pdv {p\zeta}{x_i} + \frac\kappa{c_p} \sum_{l=1}^{n_\text{sp}}\pdv{Y_l}{x_i}$.
In \cref{eq: CVRC rho Z}, the first three terms are quadratic in the velocity components and the variable $\rho Z_m$, while we note for the last three terms that the product rule gives $\pdv {\rho Z_m}x = \rho \pdv{Z_m}x + Z_m\pdv\rho x$, which gives the following identity for the $x$-derivative of $Z_m$ (and similar identities for $y$ and $z$):
\begin{align}
  \pdv {Z_m}x = \zeta \pdv{\rho Z_m}x - Z_m \zeta \pdv \rho x = \zeta \pdv{\rho Z_m}x + \rho Z_m \pdv{\zeta}x.
\end{align}
Since the first spatial derivatives of $Z_m$ are quadratic in the variable $\rho Z_m$ and the specific volume $\zeta$, the second spatial derivatives must also be quadratic in these variables (due to linearity of the derivative operator), so \cref{eq: CVRC rho Z} is quadratic in the lifted variables. A similar argument shows that \cref{eq: dCdt CVRC} is quadratic in the velocity components, $\zeta$, and $\rho C$, except for the reaction source term $\dot\omega_C$, which is non-quadratic due its dependence on the flamelet look-up table.
Finally, \cref{eq: CVRC dTdt} is approximately cubic in the specific volume variables because $\pdv pt$ and $\pdv \zeta t$ are approximately quadratic in the specific volume variables. 
Because the dynamics in the lifted variables are not exactly quadratic, we will employ in our numerical experiments a regularization penalty to combat errors due to the misspecification of the learned model.

\subsection{CVRC test problem and learning task}\label{ssec: CVRC test problem}
The CVRC geometry and operating conditions are described in \Cref{sssec: geom op cond}, and the learning task is described in \Cref{sssec: learning task}.

\subsubsection{Geometry and operating conditions}\label{sssec: geom op cond}
The CVRC geometry for our problem is a truncated segment of the full combustor studied in~\cite{harvazinski2015coupling}, and we seek to model the dynamics of the truncated domain under the same inlet operating conditions as in that work. The combustor is depicted in~\Cref{fig: CVRC geom} and is cylindrically symmetric around the $x_1$-axis. The length of the entire combustor is approximately 28 cm in the axial $x_1$-direction. 
The dynamics are driven by forcing in the form of a 10\% fluctuation in the back pressure $p_b(t)$ at the combustor outlet at the right boundary of the domain:
\begin{align}
  p_b(t) = p_{b,0} +0.1\, p_{b,0}\sin (2\pi \nu \, t),
\end{align}
where the forcing frequency is $\nu = 2000$ Hz and the baseline back pressure is $p_{b,0} = 1.1$~MPa. 
\begin{figure}[h]
  \centering
  \includegraphics[width=\textwidth]{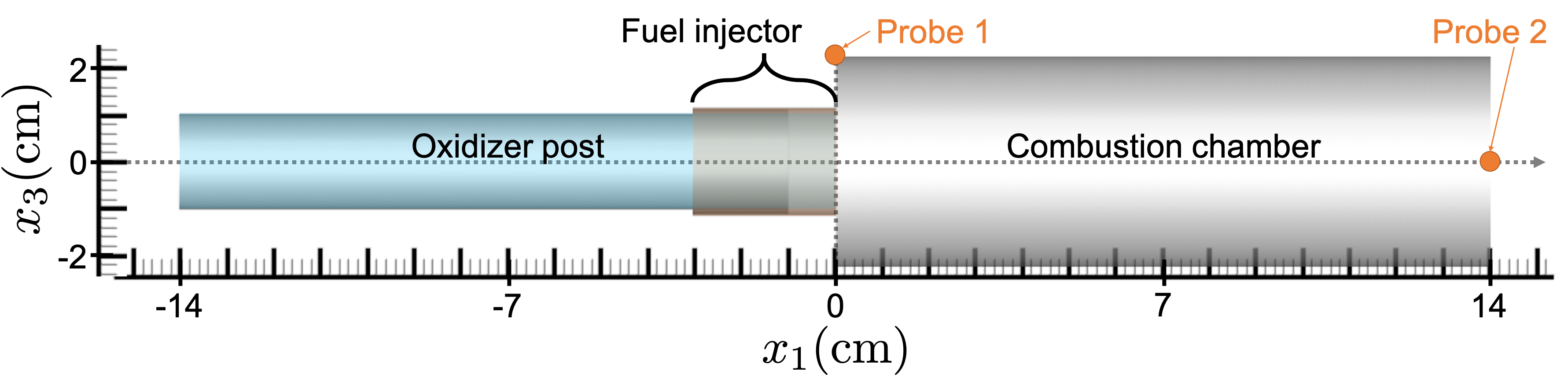}
  \caption[CVRC geometry]{CVRC geometry viewed from negative $y$-axis and locations of pressure probes. Probe 1 has coordinates $(0,0,z_{\rm max})$ and Probe 2 $(x_{\rm max},0,0)$.\label{fig: CVRC geom}}
\end{figure}

\subsubsection{Learning task and available data}\label{sssec: learning task}
Simulation data from solving the original nonlinear equations \cref{eq: CVRC cons equations} on a mesh with $n = 2,637,771$ cells is available at 5000 simulation timesteps spaced $\Delta t = 10^{-6}$ seconds apart, corresponding to 5 milliseconds of simulation time between $t=15$~ms and $t = 20$~ms. Although the conservative formulation \cref{eq: CVRC cons equations} is used to discretize the governing equations, yielding
\begin{align*}
   n\cdot d = 2,637,771 \cdot 7 = 18,464,397
\end{align*}
degrees of freedom in the simulation, snapshots are saved for each of the following variables: $\rho$, $u_1$, $u_2$, $u_3$, $p$, $Z_m$, $C$, $T$, and enthalpy $h$. The runtime to compute all 5 milliseconds of simulation time was over 45,000 CPU hours.

The learning task is to use data from the first $K = 3{\color{red},}000$ simulation time steps, from $t = 15$ ms to $t = 17.999$ ms to compute the POD basis and to train a Lift \& Learn model. Loading these 3{\color{red},}000 snapshots requires 471 GB of RAM{\color{red}, which makes the processing of this data set memory intensive}. For this example, the PDE formulation leads us to use the cell volumes of the simulation to approximate POD basis functions in the $L^2(\Omega)$-norm. We will then use the resultant learned model to predict from $t = 15$ ms to $t = 20$ ms:\ the prediction task is thus to reconstruct the dynamics of the training period \emph{and} to predict dynamics two milliseconds beyond the training period. We will then assess the accuracy of the learned model predictions in pressure, temperature, and the flamelet manifold parameters.

We use the available simulation data to compute the learning variables in~\cref{eq: CVRC lifted variables} for each of the first 3{\color{red},}000 time steps. Because the lifted variables have magnitudes that differ by up to seven orders of the magnitude---pressure has values in MPa ($\mathcal{O}(10^6)$) and $\rho C$ and $\rho Z_m$ are generally $\mathcal{O}(10^{-1})-\mathcal{O}(1)$, we center and scale the data in each lifted variable before computing a POD basis. The snapshots in each lifted variable are first centered around the mean field (over the 3{\color{red},}000 training timesteps) in that variable, and then scaled by the maximum absolute value of that variable so that the values for each lifted variable do not exceed the range $[-1,1]$.

The POD basis is computed in the $L^2(\Omega)$ norm by weighting the discrete norm by the cell volumes of the simulation data.  Singular values of the weighted data are shown in the left half of \Cref{fig: CVRC svals energy}. Let $\eta_r$ denote the relative fraction of energy retained by the $r$-dimensional POD basis,
\begin{align}
  \eta_r = 1 - \sum_{i=r+1}^{3{\color{red},}000} \sigma_i^2 \bigg/ \sum_{i=1}^{3{\color{red},}000}\sigma_i^2,
\end{align}
where $\sigma_i$ is the $i$-th singular value of the data matrix. 
This retained energy is shown on the right in \Cref{fig: CVRC svals energy}.

\begin{figure}[h]
  \centering
  \includegraphics[width=\textwidth]{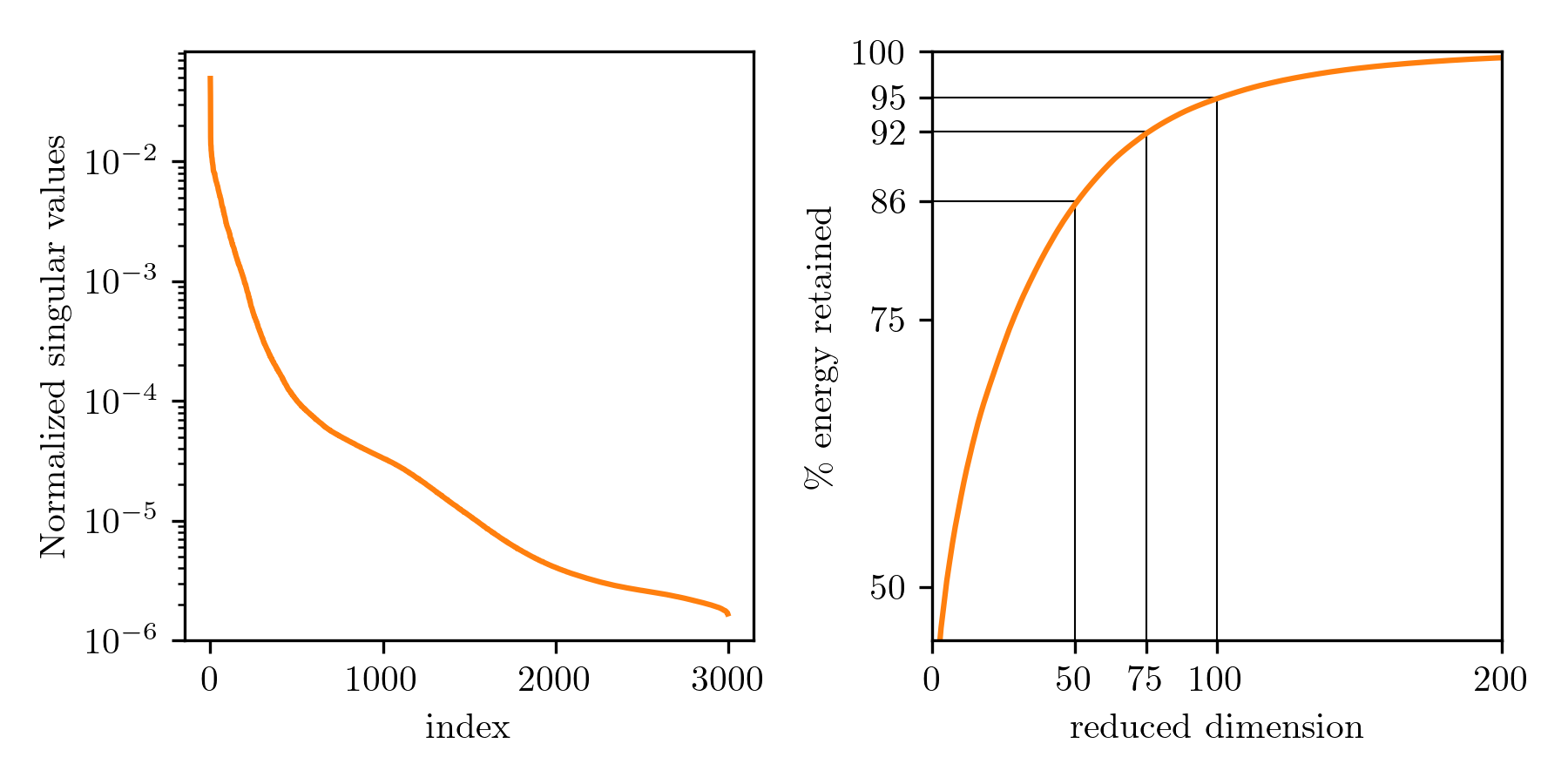}
  \caption[Singular values of CVRC training data]{\emph{Left:} Singular values of scaled and centered lifted data, normalized by the sum of all singular values. \emph{Right:} Energy retained by POD bases for the lifted CVRC data. \label{fig: CVRC svals energy}}
\end{figure}

\Cref{fig: CVRC svals energy} highlights the retained energy at $r = 50$, $r= 75$, and $r = 100$, which are the basis sizes for which we will fit a reduced model: we discuss the choice of these basis sizes in \Cref{sssec: reg op inf}. {\color{red}Due to the complexity of the dynamics of this transport-dominated reacting flow,} the decay of the singular values of the data is relatively slow: even with 100 basis functions, only 95\% of the POD energy is retained. This is lower than what is considered ideal for traditional projection-based reduced modeling (usually 99.9\% or above). However, for our 3{\color{red},}000 available snapshots, the size of the model that can be uniquely specified by the available data limits the sizes of the models we learn.

\subsection{Lift \& Learn Operator Inference formulation}\label{ssec: CVRC op inf formulation}
We now present the regularized operator inference problem used to learn a reduced model for prediction of the CVRC dynamics. \Cref{sssec: CVRC model param} describes the form of the model we learn. \Cref{sssec: reg op inf} describes the regularization strategy employed. For this problem, tuning of the regularization parameters is critical to the performance of the method, leading to similar results when the $L^2(\Omega)$ and Euclidean norms are used. The principal purpose of this example is to demonstrate the scalability of our method, so we present results solely for the $L^2(\Omega)$-norm approach that is the focus of this work.

\subsubsection{Parametrization of the learned model}\label{sssec: CVRC model param}
Because most of the terms of the governing equations in the learning variables are linear or quadratic in the learning variables, constant, or linear in the forcing, we will prescribe the following model form for the reduced model to be learned: 
\begin{align}
  \pdv{\tilde w}t = \tilde A\tilde w + \tilde H(\tilde w\otimes \tilde w) + \tilde G + \tilde B p_\nu(t),
  \label{eq: CVRC model form}
\end{align}
where the reduced lifted state $\tilde w$ has dimension $r$, and the reduced operator dimensions are given by $\tilde A\in\R^{r\times r}$, $\tilde H\in\R^{r\times r^2}$, $\tilde G\in\R^r$, and $\tilde B\in\R^r$, and the input is defined by $p_\nu(t) = p_b(t)-p_{b,0}$.

\subsubsection{Regularized operator inference minimization}\label{sssec: reg op inf}
As discussed in~\Cref{ssec: lifted CVRC eqs}, the true governing equations in the learning variables contain non-quadratic terms that are not reflected in~\cref{eq: CVRC model form}. To reduce the tendency to overfit to the unmodeled dynamics, we regularize the Operator Inference problem as follows:
\begin{align}
  \begin{split}
    \arg\min_{\tilde A,\tilde H,\tilde B,\tilde G} \bigg(\frac 1K\sum_{k=1}^K &\left\| \tilde A\tilde w_k + \tilde H(\tilde w_k\otimes\tilde w_k) + \tilde G+ \tilde B p_\nu(t_k) - \dot{\tilde w}_k\right\|_2^2 \cdots \\ 
    &\qquad +\gamma_1\left(\|\tilde A\|_F^2 + \|\tilde B\|_F^2 + \|\tilde G\|_F^2\right) + \gamma_2\|\tilde H\|_F^2\bigg),
  \end{split}
  \label{eq: CVRC reg op inf}
\end{align}
where $\gamma_1$ is the weighting of the regularization penalty on the linear, constant, and input operators, and $\gamma_2$ is the regularization penalty on the quadratic operator. We weight the quadratic operator separately because the quadratic terms have a different scaling than the linear terms. The regularization weights $\gamma_1$ and $\gamma_2$ are tuned using a procedure that minimizes the model prediction error over the training period subject to a growth constraint (see \Cref{sec: reg tuning} for additional information)~\cite{mcquarrie2021data}.

Note that even after adding regularization to the minimization in~\cref{eq: CVRC reg op inf}, the minimization still decomposes into $r$ independent least-squares problems, one for each row of the reduced operators. Once the structural redundancy of the Kronecker product $\tilde w\otimes\tilde w$ is accounted for, each of these $r$ independent least-squares problems has $r + \frac{r(r+1)}2+2$ degrees of freedom. The number of degrees of freedom of each of the independent least-squares problems for each of the model sizes we learn, and the regularization weights chosen by our tuning procedure, are tabulated in \Cref{tab: CVRC DOF gamma}. 

\begin{table}[h]
  \centering
  \begin{tabular}{c | c | c c}
    Model size $r$ & operator inference degrees of freedom & $\gamma_1$ & $\gamma_2$ \\
    \hline
    50 & 1327 & 6.95e4 &  1.62e12\\
    75 & 2927 & 3.36e4 & 3.79e11\\ 
    100 & 5152 & 2.07e2 & 1.83e11\\
    \multicolumn{4}{c}{}
  \end{tabular}
  \caption[Learned model sizes and regularization weights for CVRC simulation]{Operator inference specifications for CVRC numerical experiments. The number of degrees of freedom in operator inference is the size of each of the $r$ independent least-squares problems, $r+\frac{r(r+1)}2+2$. Note that the available data set can only fully specify up to 3{\color{red},}000 degrees of freedom without regularization. The tabulated $\gamma_1$ and $\gamma_2$ values are those chosen by the regularization tuning procedure described in \Cref{ssec: CVRC op inf formulation}.}
  \label{tab: CVRC DOF gamma}
\end{table}

\Cref{tab: CVRC DOF gamma} shows that the three model sizes that we learn represent three different regimes of the least squares Operator Inference problem. Recall that we have $K=3{\color{red},}000$ training snapshots available to us. The $r = 50$ model with 1{\color{red},}327 degrees of freedom represents the overdetermined regime, where there are many more data than degrees of freedom. The $r = 100$ model with 5{\color{red},}152 degrees of freedom represents the underdetermined regime, where there are many more degrees of freedom than data. Finally, the $r =75$ model with 2{\color{red},}927 degrees of freedom is the largest model of the form \cref{eq: CVRC model form} that can be fully specified by the 3{\color{red},}000 data.
In our numerical results, we will compare the performance of the three learned models representing these three regimes.

\subsection{Lift \& Learn model performance}\label{ssec: CVRC prediction results}
We now examine the prediction performance of the learned reduced models of sizes $r \in \{50, 75, 100\}$ in the key quantities of interest --- the pressure, temperature, and flamelet manifold parameters.
\Cref{ssec: metrics} describes the metrics we use to assess predictions in the quantites of interest and \Cref{sssec: results} presents and discusses our results.

\subsubsection{Performance metrics}\label{ssec: metrics}
The pressure, temperature, and flamelet manifold parameters are the primary quantities of interest, due to the role of pressure in combustion instability, the temperature limits of materials, and the full specification of the chemical model by the flamelet manifold parameters. 
To evaluate the Lift \& Learn pressure predictions, we will measure the predicted pressure at two point probes, whose locations are shown in \Cref{fig: CVRC geom}.
Because acoustics are global, these pointwise error measures are accurate reflections of the overall error in pressure prediction.

In contrast, the dynamics in the temperature and flamelet manifold parameters are transport-dominated, so pointwise error measures can be misleading. To assess the performance of the learned model in these variables, we instead consider the cross-correlation between the reference predicted field and the field predicted by the learned model at time $t$. This measure is informed by the use of cross-correlation in particle image velocimetry analysis of fluid flows, where the cross-correlation is used as a measure of similarity between image frames at different times~\cite{keane1992theory}. Here, we will use the cross-correlation between the fields predicted by different models \emph{at the same time} as a measure of the similarity between predictions.

That is, let $\bT^{\rm ref}(t)\in\R^n$ and $\bT^{\rm L\&L}(t)\in\R^n$ denote the reference and learned model discrete temperature predictions, respectively, and define
\begin{align}
  \label{eq: T correlation}
  R_T(t) &= \frac{\sum_{i=1}^n (\bT^{\rm ref}_i(t) - \mu_\bT^{\rm ref}(t)) (\bT^{\rm L\& L}_i(t) - \mu_\bT^{\rm L\&L}(t))}{\sigma_\bT^{\rm ref} \sigma_\bT^{\rm L\&L}}, 
\end{align}
where $\mu_\bT^{\rm ref}(t)\in\R$ and $\mu_\bT^{\rm L\&L}(t)\in\R$ are the mean values, respectively, and $\sigma_\bT^{\rm ref}(t)\in\R$ and $\sigma_\bT^{\rm L\&L}(t)\in\R$ are the standard deviations, respectively, of $\bT^{\rm ref}(t)\in\R^n$ and $\bT^{\rm L\&L}(t)\in\R^n$, respectively. This measure $R_T(t)$ is the Pearson correlation coefficient between the reference and learned model predictions of temperature at time $t$.
We define the correlation measures $R_{Z_m}(t)$ and $R_{C}(t)$ in a similar way. 
By definition, the correlations $R_*(t)$ take on values in the range $[-1,1]$, where $R_*(t)=1$ indicates perfect correlation and $R_*(t)=-1$ perfect anti-correlation, with $R_*(t)=0$ indicating no correlation.

\subsubsection{Results and discussion}\label{sssec: results}
The regularized operator inference minimization in~\cref{eq: CVRC reg op inf} is solved for $r = 50$, $r = 75$, and $r = 100$, with the regularization weights given in \Cref{tab: CVRC DOF gamma}. The inferred reduced operators define a model of the form \cref{eq: CVRC model form}. This inferred model is then numerically integrated using a second-order explicit Runge-Kutta time-stepping scheme to predict the combustor dynamics from the initial condition at $t = 15$~ms. 

\Cref{fig: probe both traces,fig: CVRC temperature snaps,fig: CVRC flamelet snaps,fig: CVRC progress snaps} give an overview of the learned model prediction performance for the primary quantities of interest. \Cref{fig: probe both traces} shows the pressure measured at the two probes, while \Cref{tab: CVRC pressure errors} reports $L^2(\Omega)$-errors in the predicted pressure fields. \Cref{fig: CVRC temperature snaps,fig: CVRC flamelet snaps,fig: CVRC progress snaps} show cross-sections of the temperature, flamelet mixture mean, and reaction progress fields taken at the $x_2=0$ plane of the combustor at each millisecond mark. Field reconstruction is done by multiplying the POD basis vectors by their reduced state coefficients and post-processing the reconstructed states to eliminate non-physical values, i.e., temperatures below 0 were set to 0 and flamelet model parameters outside the range $[0,1]$ were set to the nearest value in the range. Non-physical predictions are not unusual in projection-based reduced modeling, since the reduced models evolve coefficients of basis functions rather than the physical state variables themselves. Formulations that embed variable range constraints within the Operator Inference framework are a possible direction of future work, following the ideas proposed in~\cite{huang2019investigations}.

\begin{figure}[h]
  \centering
  \includegraphics[width=\textwidth]{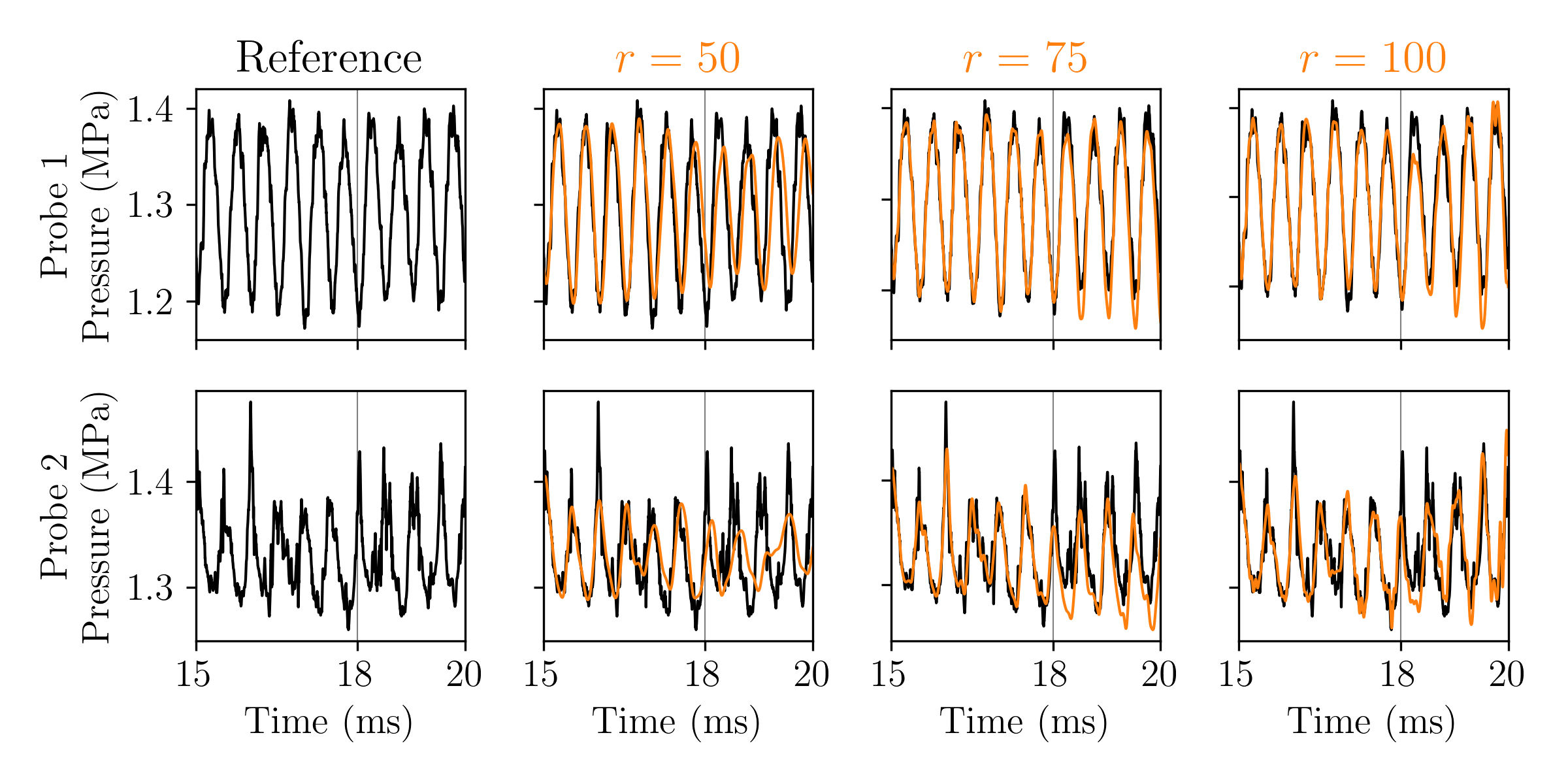}
  \caption[CVRC pressure prediction at probe locations]{Pressure predictions at probe locations. Probe 1 has coordinates $(0,0,x_{3,\rm max})$ and Probe 2 has coordinates $(x_{1,\rm max},0,0)$. Training period ends at $t = 18$ ms.}
  \label{fig: probe both traces}
\end{figure}

\begin{table}[h]
  \centering
  \begin{tabular}{c | c  c}
    & \multicolumn{2}{c}{Pressure error} \\
    Model size $r$ & Reconstruction & Prediction\\
    \hline
    50 & 0.015 & 0.032\\
    75 & 0.007 & 0.018\\ 
    100 & 0.008 & 0.017
  \end{tabular}
  \caption{Mean relative $L^2(\Omega)$-errors in pressure predictions of the Operator Inference learned models during the $t\in[15, 18)$ reconstruction period and the $t\in[18, 20]$ prediction period.}
  \label{tab: CVRC pressure errors}
\end{table}

The predicted pressure fields of all three learned reduced models demonstrate good agreement with the reference data of the original high-dimensional simulation, with reconstruction errors around 1-2\% and generalization errors around 2-3\% in the prediction phase (\Cref{tab: CVRC pressure errors}). In the pressure traces in \Cref{fig: probe both traces}, we observe excellent agreement between the learned reduced models and the original simulation in the dominant frequency and amplitude of the pressure oscillations during the reconstruction phase, with agreement in amplitude deteriorating slightly in the prediction beyond 18ms. 

\begin{figure}
  \includegraphics[width=\textwidth]{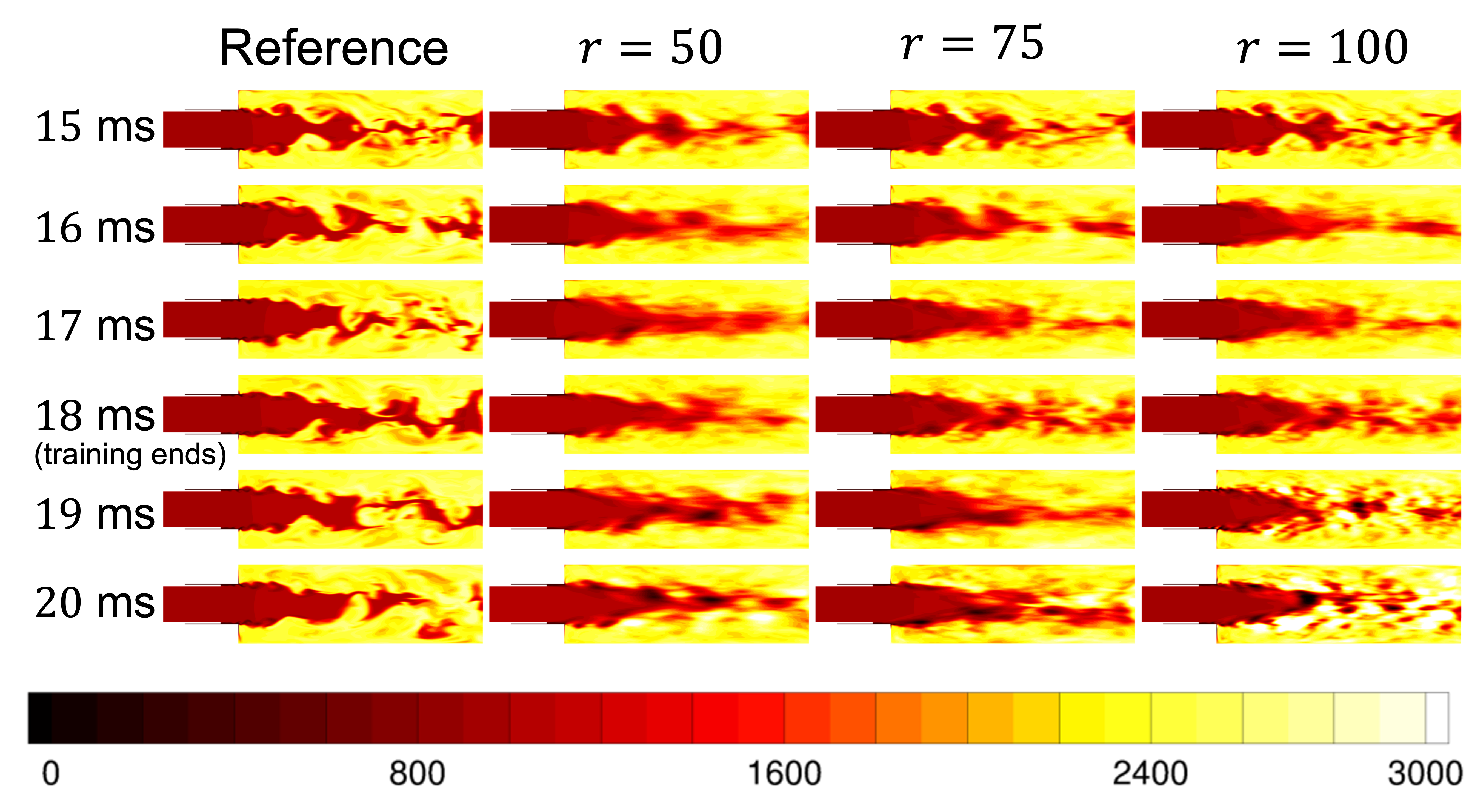}
  \caption[CVRC temperature predictions]{Temperature (Kelvin) field predictions at $x_2=0$ plane}
  \label{fig: CVRC temperature snaps}
\end{figure}

\begin{figure}
  \includegraphics[width=\textwidth]{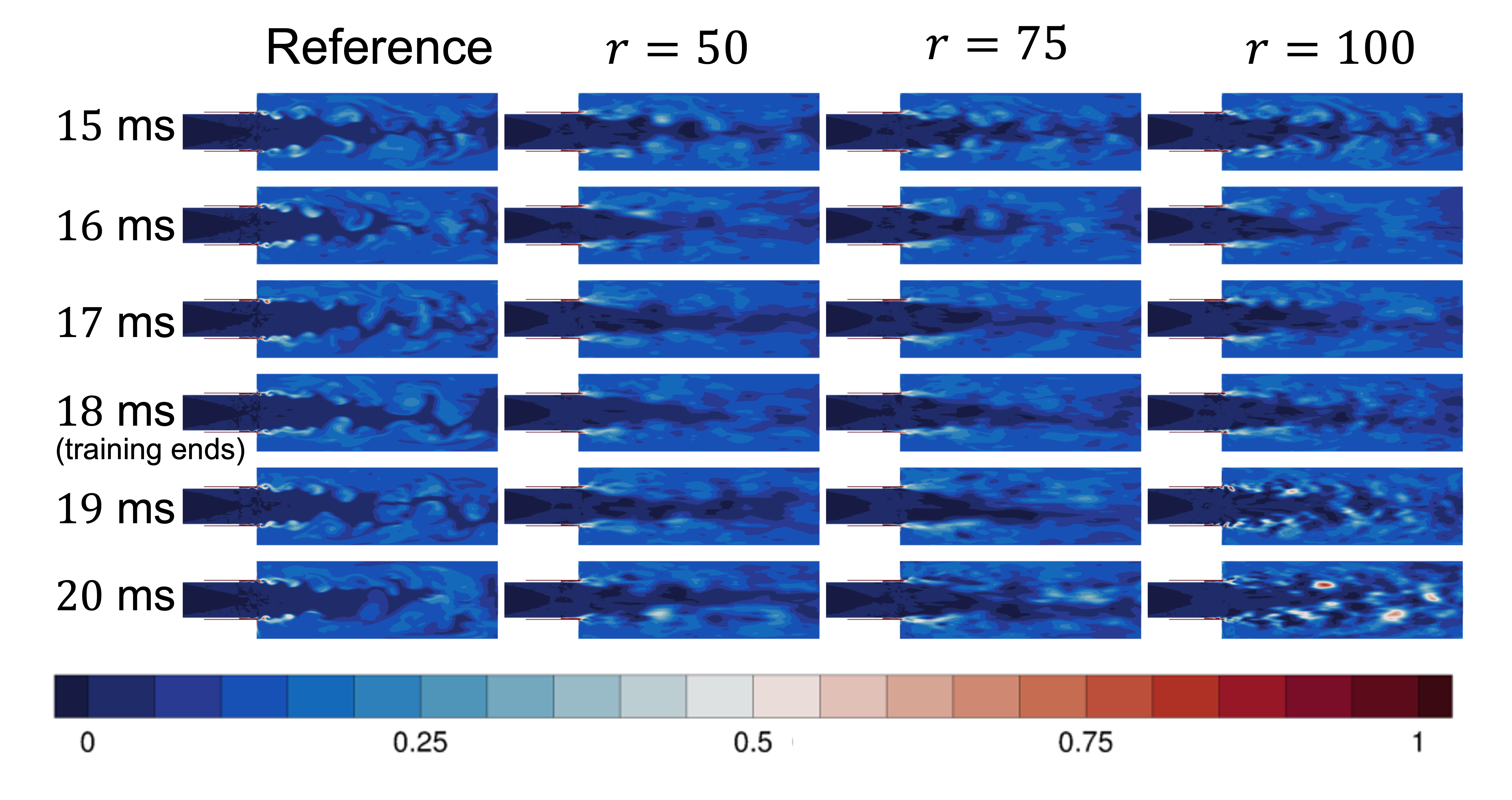}
  \caption[CVRC flamelet mixture mean predictions]{Flamelet mixture mean ($Z_m$) field predictions at $x_2=0$ plane}
  \label{fig: CVRC flamelet snaps}
\end{figure}

\begin{figure}
  \includegraphics[width=\textwidth]{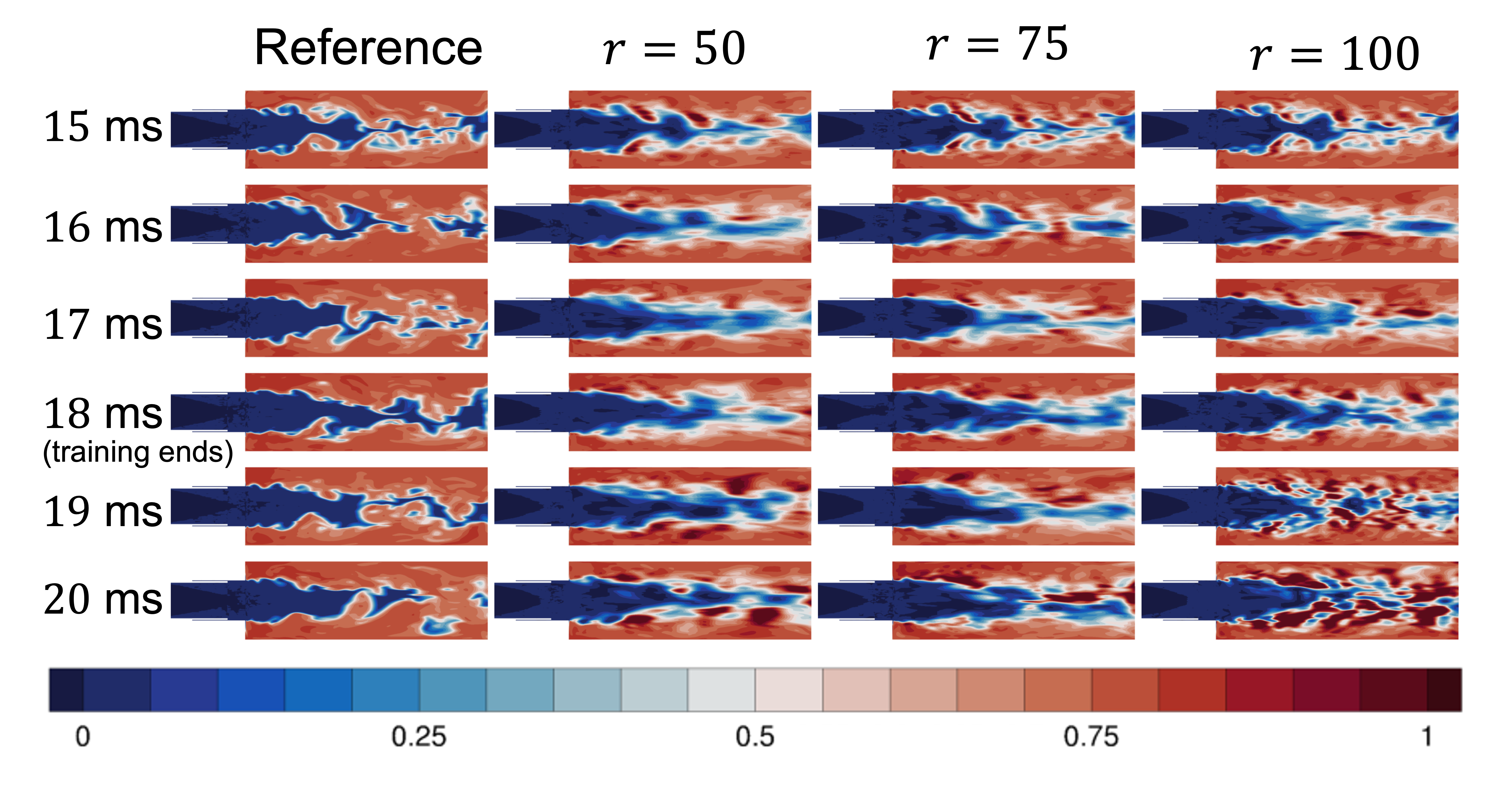}
  \caption[CVRC reaction progress predictions]{Reaction progress variable ($C$) field predictions at $x_2=0$ plane}
  \label{fig: CVRC progress snaps}
\end{figure}

The cross sections of the temperature and combustion parameter fields in \Cref{fig: CVRC temperature snaps,fig: CVRC flamelet snaps,fig: CVRC progress snaps} qualitatively illustrate the tradeoff inherent in choosing a basis size $r$ for learned a reduced model. The largest $r=100$ case is the most expressive, which can be seen by comparing the initial $t=15$ ms snapshots which simply project the full reference data onto the basis. The $t = 15$ ms snapshots for the $r = 100$ case capture the most structure of the fields, including smaller-scale structures. The $r = 75$ case captures fewer smaller-scale structure and the $r = 50$ case smears out all but the coarser flow features. However, due to the limited available data, the $r= 100$ model cannot be fully specified by the data, and the dynamics of the learned model that results from regularization lead to snapshots at later times that exhibit poorer agreement with the reference simulation (compare for example the $t = 16$ ms temperature snapshots between $r = 100$ and $r = 75$). Another drawback of the underspecified $r = 100$ learning problem is that the resulting model predicts several regions where the temperature and flamelet model parameters reach their extremal values, especially at the final $t = 20$ ms (note regions of temperature at 0 K and 3{\color{red},}000 K, well outside the range of the reference).  These extreme predictions in the $r =100$ case illustrate the pitfalls of attempting to fit a model to less data than there are degrees of freedom in the model. While the $r = 50$ and $r=75$ models exhibit some extreme predictions, these extreme regions are more limited, illustrating the importance of choosing a model size $r$ for which the data allow the learning problem to be fully specified. Building multiple localized Operator Inference reduced models instead of one global reduced model is one way to address this issue, since the dimension of each local reduced model can be kept small~\cite{geelen2021localized}.

The trade-off between the expressivity of larger $r$ and ability at smaller $r$ to learn accurate dynamics from limited data is illustrated quantitatively in \Cref{fig: TCZ corr coeff}, which plots the cross-correlation metric defined in \Cref{ssec: metrics} for temperature and the flamelet model parameters. These correlation measures show that while all three learned reduced models yield field predictions that are well-correlated with the reference simulation data during the $t<18$ ms reconstruction regime, correlation deteriorates sharply for the $t>18$ ms prediction phase for the larger two model sizes, $r = 75$ and $r = 100$. This quantitative measure illustrates clearly that the $r = 50$ model achieves the best generalization performance in the sense that its reconstruction and prediction correlations are similar, avoiding the sharp decline at the transition to prediction of the other models. The poorer generalization performance of the larger two models is likely the result of the both limited data and model misspecification -- as described in \Cref{sssec: CVRC model param}, the lifted governing equations of the CVRC contain mostly linear and quadratic terms that are reflected in the form of the reduced model, but not all terms have this form. The Operator Inference learning problem for $r = \{75,100\}$ may be more sensitive to non-quadratic dynamics reflected in the data, leading to poor generalization performance. In contrast, in the $r = 50$ case there is sufficient data to avoid overfitting in the more limited number of degrees of freedom.

\begin{figure}[h]
  \centering
  \includegraphics[width=0.97\textwidth]{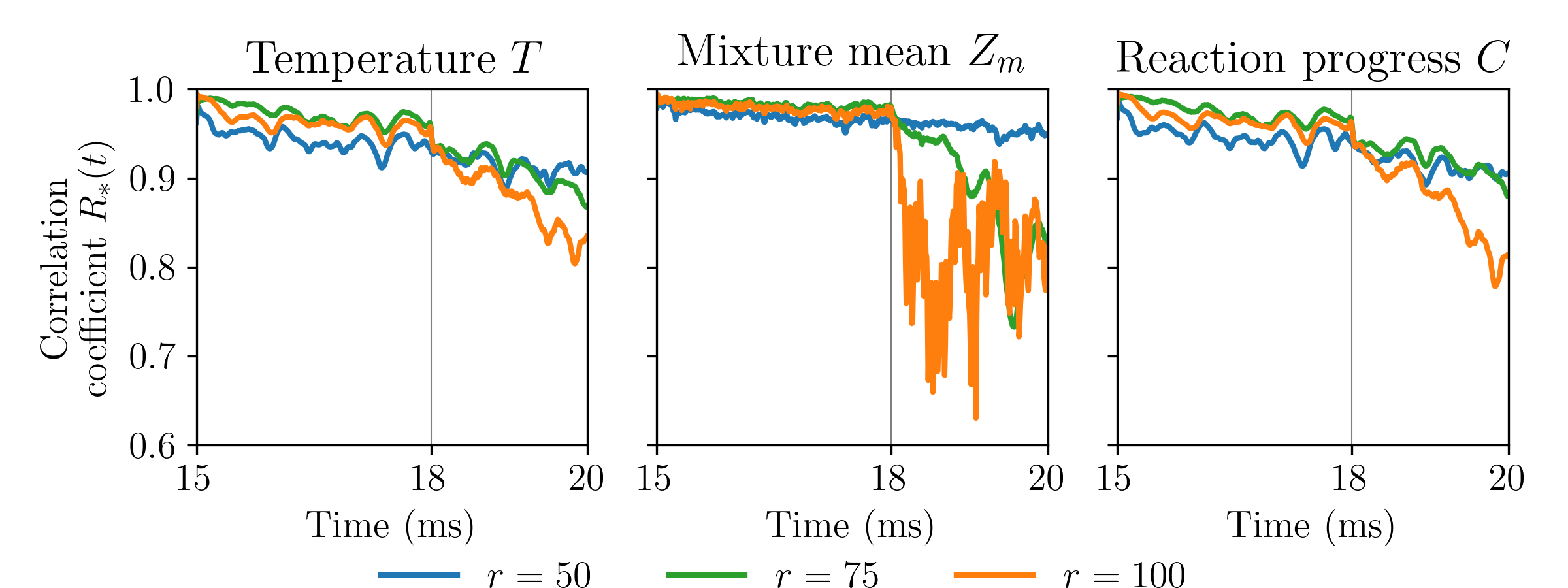}
  \caption[Cross-correlation of Lift \& Learn temperature and flamelet parameter predictions]{Cross-correlation between learned model prediction and reference data for temperature, flamelet mixture mean, and reaction progress variable. Training data ends at $t=18$~ms.}
  \label{fig: TCZ corr coeff}
\end{figure}

Overall, while the learned reduced models fail to capture some aspects of the CVRC dynamics, including some of the high-frequency pressure oscillations and smaller-scale structures of the flow fields, when there is sufficient data for the model size, the learned model can accurately predict large-scale structures (as measured by correlation coefficients above 0.8) and the dominant amplitude and frequency of the pressure oscillations. Our learned models achieve this prediction ability with a dimension reduction of five orders of magnitude relative to that of the reference simulation. This allows the learned reduced model to simulate 5 milliseconds of simulation time in a matter of a few seconds of run time, while the high-dimensional simulation that generated the data set required approximately 45{\color{red},}000 CPU hours to generate. This dimension reduction is made possible by our \emph{physics-informed} model learning framework, which transforms the governing equations to mostly-quadratic form, fits a quadratic reduced model, and regularizes against the model's misspecification of the remaining non-quadratic terms.

\section{Conclusions}\label{sec: conclusions}
We have presented Operator Inference for PDEs, a new formulation for scientific machine learning which learns reduced models for systems governed by nonlinear PDEs by parametrizing the model learning problem by low-dimensional polynomial operators which reflect the known polynomial structure of the governing PDE. Operator Inference can be applied to systems governed by more general nonlinear PDEs through Lift \& Learn, which uses lifting transformations to expose quadratic structure in the governing PDE. 
By building this structure due to the governing physics into the learned model representation, our scientific machine learning method reduces the need for massive training data sets when compared to generic representations such as neural networks.

Our Operator Inference formulation generalizes the Operator Inference method previously developed in~\cite{Peherstorfer16DataDriven} for systems of ODEs, allowing the learned model parametrization to be determined based solely on the form of the underlying governing physics rather than on the form of the discretization of the PDE. The formulation in the PDE setting learns a mapping between low-dimensional subspaces of the underlying infinite-dimensional Hilbert space, allowing more general variable transformations between Hilbert spaces to be considered, in contrast to~\cite{QKPW19LiftLearn} which restricts the class of acceptable variable transformations to pointwise transformations. A numerical demonstration for a heat equation problem on a non-uniform grid illustrates that the formulation in the PDE setting yields an algorithm that finds basis functions consistent with the underlying continuous truth and can lead to a lower error than an ODE-based formulation.

{\color{red}We demonstrate the potential of the proposed method to scale to problems of high-dimension by learning} reduced operators for a three-dimensional combustion simulation with over 18 million degrees of freedom that requires 45{\color{red},}000 CPU hours to simulate {\color{red}five} milliseconds of simulation time. The resulting reduced models accurately predict the amplitude and frequency of pressure oscillations, a key factor in the development of combustion instabilities, with just 50 reduced states, a five order of magnitude dimension reduction. The reduced model can simulate {\color{red}five} milliseconds of simulation time in just a few seconds, a speed-up which makes the learned reduced model suitable for use in the many-query computations which support engineering decision making.

\section*{Acknowledgments}
The authors are grateful for Cheng Huang and Chris Wentland for the data used in the CVRC numerical experiments as well as for several helpful discussions.
This work was supported in parts by the National Science Foundation Graduate Research Fellowship Program, the Fannie and John Hertz Foundation, the US Air Force Center of Excellence on Multi-Fidelity Modeling of Rocket Combustor Dynamics award FA9550-17-1-0195, the Air Force Office of Scientific Research MURI program awards FA9550-15-1-0038,  FA9550-18-1-0023 and FA9550-21-1-0084, the US Department of Energy Applied Mathematics MMICC Program award DE-SC0019303, and the SUTD-MIT International Design Centre. 

\appendix
\section{Regularization tuning strategy}\label{sec: reg tuning}
To choose the regularization weights $\gamma_1$ and $\gamma_2$, we employ a variant of the regularization tuning procedure of~\cite{mcquarrie2021data}. That is, we use the regularization weights that minimize (on a two-dimensional grid) the training error, subject to a constraint that the resultant reduced model have bounded growth within a trial integration period. In more detail, we solve the minimization \cref{eq: CVRC reg op inf} on a $40\times 40$ grid of values spaced log-uniformly in $(\gamma_1,\gamma_2)\in[10^2, 10^8]\times[10^{9},10^{14}]$. The resultant reduced model is then integrated from $t = 15$ ms to $t = 22$ ms (two milliseconds \emph{beyond} the desired prediction time), and the final regularization parameters are chosen to be the regularization parameters that minimize the error in the predicted reduced coefficients over the training period, subject to the following growth constraint:
\begin{align}
  \label{eq: regularization growth constraint}
  \max_{\substack{i\leq r \\ k \leq 7000}}\left|\hat\bw_i^{\rm trial}(t_k) - \bar\bw_i\right| \leq 1.2\max_{\substack{i\leq r\\ k \leq 3000}}\left| \hat\bw_i^{\rm training}(t_k) - \bar\bw_i\right|,
\end{align}
where $\bar\bw_i$ is the mean of the $i$-th reduced coefficient over the 3000 training time steps, $\hat\bw_i^{\rm trial}(t_k)$ is the value of the $i$-th reduced coefficient predicted by the inferred model at time $t_k$, and $\hat\bw_i^{\rm training}(t_k)$ is the value of the $i$-th reduced coefficient in the training data at time $t_k$. In words, \cref{eq: regularization growth constraint} requires that when the reduced model is integrated for the full 7 ms period, no single reduced coefficient deviates more from its training mean more than 20\% more than the maximum deviation from the training mean during the training period. We seek models that satisfy this constraint to avoid the selection of regularization parameters that may predict the dynamics of the training period well but become unstable beyond the training period.

We note that all computations in this tuning strategy---the regularized minimization in~\cref{eq: CVRC reg op inf}, the integration of the resulting reduced model, and the computation of the error and growth in the reduced coefficients---are dependent only on $r$-dimensional quantities. This strategy for tuning the regularization parameters is therefore computationally inexpensive. The regularization weights chosen by this tuning strategy for the three model sizes we test are tabulated in \Cref{tab: CVRC DOF gamma}.

\section{Derivation of lifted CVRC governing PDEs }\label{appendix}
\subsection{Derivation of specific volume equation}
Recall that
\begin{align}
  \pdv \rho t = -\nabla \cdot (\rho u) = -\pdv{\rho u_x}x - \pdv{\rho u_y}y - \pdv{\rho u_z}z.
\end{align}
Since $\zeta = \frac1\rho$, applying the chain rule yields
\begin{align}
  \pdv{\zeta}t &= \pdv{}t\frac1\rho = -\frac1{\rho^2}\pdv\rho t = \frac1{\rho^2} \left(\pdv{\rho u_x}x + \pdv{\rho u_y}y + \pdv{\rho u_z}z\right) \nonumber \\
  &=\frac1{\rho^2}\left(\pdv\rho x u_x + \rho \pdv{u_x}x + \pdv\rho y u_y + \rho \pdv{u_y}y + \pdv\rho z u_z + \rho \pdv{u_z}z \right) \nonumber \\
  &= -\frac1{\rho^2}\pdv \rho x(-u_x) -\frac1{\rho^2}\pdv \rho y(-u_y) -\frac1{\rho^2}\pdv \rho z(-u_z) +\frac1\rho \left(\pdv{u_x}x + \pdv{u_y}y + \pdv{u_z}z\right)\nonumber \\
  &= -\pdv\zeta x u_x -\pdv\zeta y u_y -\pdv \zeta z u_z +\zeta(\nabla\cdot u) = -\nabla\zeta \cdot u + \zeta(\nabla\cdot u).
\end{align}

\subsection{Derivation of velocity equations}
We derive the equation for the evolution of the $x$-velocity $u_x$ from the equations for the $x$-momentum $\rho u_x$ and for the density:
\begin{align}
  \pdv{u_x}t &= \frac1\rho \left(\pdv{\rho u_x}t - \pdv{\rho}t u_x\right)  \nonumber \\
  & =\frac1\rho \bigg(-\pdv{\rho u_x^2 + p}x - \pdv{\rho u_x u_y}y - \pdv{\rho u_x u_z}z + \pdv{\tau_{xx}}x + \pdv{\tau_{xy}}y + \pdv{\tau_{xz}}z \cdots \nonumber \\
  & \qquad\qquad\qquad+ u_x\pdv{\rho u_x}x + u_x \pdv{\rho u_y}y + u_x\pdv{\rho u_z}z\bigg) \nonumber\\
  & = \frac1\rho\bigg( -\pdv px - u_x^2\pdv\rho x -\rho \pdv{u_x^2}x - u_xu_y\pdv\rho y -\rho\pdv{u_xu_y}y - u_xu_z\pdv\rho z - \rho \pdv{u_xu_z}z \cdots\nonumber \\
  & \qquad\qquad+ u_x^2\pdv\rho x + \rho u_x\pdv {u_x }x+ u_xu_y\pdv\rho y +\rho u_x\pdv {u_y}y + u_x u_z \pdv \rho z + \rho u_x \pdv{u_z}z\cdots \nonumber\\
  &\qquad\qquad\qquad+\pdv{\tau_{xx}}x + \pdv{\tau_{xy}}y + \pdv{\tau_{xz}}z\bigg) \nonumber\\
  & = \frac1\rho\bigg(-\pdv px - \rho \pdv{u_x^2}x - \rho\pdv{u_x u_y}y - \rho \pdv{u_x u_z}z + \rho u_x\left(\pdv {u_x} x + \pdv {u_y}y+\pdv{u_z}z\right)\cdots\nonumber\\
  & \qquad\qquad +\pdv{\tau_{xx}}x + \pdv{\tau_{xy}}y + \pdv{\tau_{xz}}z\bigg) \nonumber\\
  &=-\zeta \pdv px -\pdv{u_x^2}x - \pdv{u_x u_y}y -\pdv{u_x u_z}z + u_x\left(\pdv {u_x} x + \pdv {u_y}y+\pdv{u_z}z\right)\cdots \nonumber \\
  &\qquad\qquad +\zeta\left(\pdv{\tau_{xx}}x + \pdv{\tau_{xy}}y + \pdv{\tau_{xz}}z\right)\nonumber \\
  & = -\zeta\pdv px - 2u_x\pdv{u_x}x - u_x\pdv{u_y}y - u_y\pdv{u_x}y - u_x\pdv{u_z}z - u_z\pdv{u_x}z \cdots\nonumber\\
  &\qquad\qquad + u_x\left(\pdv {u_x} x + \pdv {u_y}y+\pdv{u_z}z\right)+\zeta\left(\pdv{\tau_{xx}}x + \pdv{\tau_{xy}}y + \pdv{\tau_{xz}}z\right) \nonumber\\
  &= -\zeta\pdv px - u_x\pdv{u_x}x - u_y\pdv{u_x}y - u_z\pdv{u_x}z+\zeta\left(\pdv{\tau_{xx}}x + \pdv{\tau_{xy}}y + \pdv{\tau_{xz}}z\right)\nonumber\\
  &= -\zeta\pdv px - u\cdot\nabla u_x +\zeta\left(\pdv{\tau_{xx}}x + \pdv{\tau_{xy}}y + \pdv{\tau_{xz}}z\right).
  \label{eq: final vel eq}
\end{align}
Since the stresses $\tau_{ij}$ are linear in the velocities, \cref{eq: final vel eq} contains only quadratic terms in the specific volume, pressure, and velocity components. By a similar derivation, we obtain analogous quadratic expressions for the velocities in the $y$- and $z$-directions:
\begin{align}
  \pdv{u_y}t &= -\zeta\pdv py -u\cdot\nabla u_y+ \zeta\left(\pdv{\tau_{yx}}x + \pdv{\tau_{yy}}y + \pdv{\tau_{yz}}z\right),\\
  \pdv{u_z}t &= -\zeta\pdv pz -u\cdot\nabla u_z+ \zeta\left(\pdv{\tau_{zx}}x + \pdv{\tau_{zy}}y + \pdv{\tau_{zz}}z\right).
\end{align}

\subsection{Derivation of pressure equation}
To derive an evolution equation for the pressure $p$, we use the simplified state equation (which assumes constant $c_p$):
\begin{align}
  e = c_p T + \frac12 \left(u_x^2 + u_y^2 + u_z^2\right) - \frac p\rho,
\end{align}
or, equivalently, using the relationship $R =\frac{\gamma-1}\gamma c_p$ and the ideal gas law,
\begin{align}
  p = (\gamma-1) \rho e - \frac{\gamma-1}2\rho \left(u_x^2 + u_y^2 + u_z^2\right),
\end{align}
where $\gamma$ is the heat capacity ratio (normally a function of $T$ and the species mass fractions $Y_l$, but assumed constant here).
Then,
\begin{align}
  \pdv pt = (\gamma-1)\left(\pdv{\rho e}t - \frac12\pdv{}t \left( \rho\left(u_x^2 + u_y^2 + u_z^2\right) \right)\right)
  \label{eq: dpdt int 1}
\end{align}
We consider $\pdv{\rho e}t$ term first, using the simplified state equation to express the energy in terms of the pressure, density, and velocities:
\begin{align}
  \pdv{\rho e}t &= -\pdv{}x \left(u_x(\rho e + p)\right) - \pdv{}y \left(u_y(\rho e + p)\right) - \pdv{}z \left(u_z(\rho e + p)\right) + \mathfrak{f}(\tau,j)\nonumber \\
  &=-\pdv{}x\left(u_x\left(\frac{\gamma}{\gamma-1}p + \frac{1}2\rho |u|^2\right)\right) - \pdv{}y\left(u_y\left(\frac{\gamma}{\gamma-1}p + \frac{1}2\rho |u|^2\right)\right) \cdots \nonumber \\
  &\qquad \qquad- \pdv{}z\left(u_z\left(\frac{\gamma}{\gamma-1}p + \frac{1}2\rho |u|^2\right)\right)+ \mathfrak{f}(\tau,j)\nonumber\\
  &= -\frac\gamma{\gamma-1}\left(\pdv{u_xp}x + \pdv{u_yp}y + \pdv{u_zp}z\right) - \frac{1}2\bigg(\pdv{u_x\rho (u_x^2 + u_y^2 + u_z^2)}x  \cdots \nonumber\\
  &\qquad \qquad + \pdv{u_y\rho (u_x^2 + u_y^2 + u_z^2)}y+ \pdv{\rho u_z(u_x^2 + u_y^2 + u_z^2)}z\bigg) + \mathfrak{f}(\tau,j), \label{eq: drhoedt}
\end{align}
where $\mathfrak{f}(\tau,j)$ contains the viscous and diffusive heat flux terms, given by:
\begin{align}
  \begin{split}
    \mathfrak{f}(\tau,j) &= \pdv{}x(u_x\tau_{xx} + u_y\tau_{xy} + u_z\tau_{xz}) + \pdv{}y(u_x\tau_{xy} + u_y\tau_{yy} + u_z\tau_{yz}) \cdots \\
    &\qquad+ \pdv{}z(u_x\tau_{xz} + u_y\tau_{yz} + u_z\tau_{zz}) + \pdv{q_x}x + \pdv{q_y}y + \pdv{q_z}z.
  \end{split}
\end{align}
We now consider the kinetic energy contribution from the $x$-velocity in \cref{eq: dpdt int 1}:
\begin{align}
  \pdv{}t \left(\frac12\rho u_x^2\right) & = \frac12 u_x^2 \pdv \rho t + \frac12\rho \pdv{u_x^2}t = \frac12 u_x^2 \left(-\pdv{\rho u_x}x - \pdv{\rho u_y}y - \pdv{\rho u_z}z\right)+  \rho u_x \pdv{u_x}t \nonumber\\
  &= -\frac12 u_x^2 \left(\pdv{\rho u_x}x + \pdv{\rho u_y}y + \pdv{\rho u_z}z\right)+  \rho u_x \bigg(-\xi \pdv px - u_x \pdv{u_x}x - u_y \pdv{u_x}y\cdots\nonumber \\
  &\qquad\qquad\qquad - u_z\pdv{u_x}z + \xi \left(\pdv{\tau_{xx}}x + \pdv{\tau_{xy}}y + \pdv{\tau_{xz}}z\right) \bigg) \nonumber\\
  &=-\frac12 u_x^2 \left(\pdv{\rho u_x}x + \pdv{\rho u_y}y + \pdv{\rho u_z}z\right)-
   u_x \pdv px -u_x \rho u_x \pdv{u_x}x -u_x \rho u_y \pdv{u_x}y \cdots\nonumber \\
  &\qquad\qquad\qquad-u_x\rho u_z\pdv{u_x}z + u_x\left(\pdv{\tau_{xx}}x + \pdv{\tau_{xy}}y + \pdv{\tau_{xz}}z\right) \nonumber \\
  &=-\frac12 u_x^2 \left(\pdv{\rho u_x}x + \pdv{\rho u_y}y + \pdv{\rho u_z}z\right)-u_x \pdv px -\frac12\rho u_x \pdv{u_x^2}x - \frac12\rho u_y \pdv{u_x^2}y \cdots\nonumber\\
  &\qquad\qquad\qquad- \frac12\rho u_z \pdv{u_x^2}z + u_x\left(\pdv{\tau_{xx}}x + \pdv{\tau_{xy}}y + \pdv{\tau_{xz}}z\right) \nonumber\\
  &= -\frac12\pdv{\rho u_x^3}x -\frac12\pdv{\rho u_y u_x^2}y - \frac12\pdv{\rho u_z u_x^2}z - u_x \pdv px + u_x\left(\pdv{\tau_{xx}}x + \pdv{\tau_{xy}}y + \pdv{\tau_{xz}}z\right). \label{eq: dKExdt}
\end{align}
We get similar expressions for the $\pdv{}t\frac12\rho u_y^2$ and $\pdv{}t\frac12\rho u_z^2$ terms:
\begin{align}
  \pdv{}t \frac12\rho u_y^2 &= -\frac12\pdv{\rho u_x u_y^2}x - \frac12\pdv{\rho u_y^3}y - \frac12\pdv{\rho u_z u_y^2}z- u_y \pdv py + u_y\left(\pdv{\tau_{xy}}x + \pdv{\tau_{yy}}y + \pdv{\tau_{yz}}z\right), \label{eq: dKEydt} \\
  \pdv{}t \frac12\rho u_z^2 &= -\frac12\pdv{\rho u_x u_z^2}x - \frac12\pdv{\rho u_yu_z^3}y - \frac12\pdv{\rho u_z^3}z - u_z \pdv pz + u_z\left(\pdv{\tau_{xz}}x + \pdv{\tau_{yz}}y + \pdv{\tau_{zz}}z\right). \label{eq: dKEzdt}
\end{align}
Subtracting \cref{eq: dKExdt,eq: dKEydt,eq: dKEzdt} from \cref{eq: drhoedt} we get
\begin{align}
  \pdv{\rho e}t - \frac12&\pdv{}t(\rho (u_x^2 + u_y^2 + u_z^2)) \nonumber\\
  &= -\frac\gamma{\gamma-1}\left(\pdv{u_xp}x + \pdv{u_yp}y + \pdv{u_zp}z\right) +u_x\pdv px +u_y \pdv py + u_z\pdv pz + \mathfrak{f}(\tau,j)\cdots\nonumber \\
  & \qquad\qquad -u_x\left(\pdv{\tau_{xx}}x + \pdv{\tau_{xy}}y + \pdv{\tau_{xz}}z\right) - u_y\left(\pdv{\tau_{xy}}x + \pdv{\tau_{yy}}y + \pdv{\tau_{yz}}z\right)\cdots\nonumber\\
  &\qquad\qquad- u_z\left(\pdv{\tau_{xz}}x + \pdv{\tau_{yz}}y + \pdv{\tau_{zz}}z\right) \nonumber\\
  &  = -\frac\gamma{\gamma-1} p\left(\pdv{u_x}x  + \pdv{u_y}y + \pdv{u_z}z\right) - \frac{1}{\gamma-1}\left(u_x\pdv px+ u_y\pdv py + u_z\pdv pz\right) \cdots\nonumber \\
  &\qquad\qquad + \mathfrak{f}(\tau,j)- u_x\left(\pdv{\tau_{xx}}x + \pdv{\tau_{xy}}y + \pdv{\tau_{xz}}z\right) \cdots\nonumber \\
  &\qquad\qquad- u_y\left(\pdv{\tau_{xy}}x + \pdv{\tau_{yy}}y + \pdv{\tau_{yz}}z\right) - u_z\left(\pdv{\tau_{xz}}x + \pdv{\tau_{yz}}y + \pdv{\tau_{zz}}z\right)\nonumber\\
  &=-\frac\gamma{\gamma-1} p (\nabla\cdot u) - \frac1{\gamma-1} u\cdot \nabla p + \mathfrak{f}(\tau, j) - u_x\left(\pdv{\tau_{xx}}x + \pdv{\tau_{xy}}y + \pdv{\tau_{xz}}z\right)\cdots\nonumber\\
  &\qquad\qquad - u_y\left(\pdv{\tau_{xy}}x + \pdv{\tau_{yy}}y + \pdv{\tau_{yz}}z\right) - u_z\left(\pdv{\tau_{xz}}x + \pdv{\tau_{yz}}y + \pdv{\tau_{zz}}z\right).
\end{align}
This gives the following expression for $\pdv pt$:
\begin{align}
  \begin{split}
    \pdv pt &= - \gamma p(\nabla\cdot u) - (u\cdot\nabla p) - u_x\left(\pdv{\tau_{xx}}x + \pdv{\tau_{xy}}y + \pdv{\tau_{xz}}z\right)\cdots\\
    &\qquad - u_y\left(\pdv{\tau_{xy}}x + \pdv{\tau_{yy}}y + \pdv{\tau_{yz}}z\right) - u_z\left(\pdv{\tau_{xz}}x + \pdv{\tau_{yz}}y + \pdv{\tau_{zz}}z\right)\cdots\\ 
    & \qquad+ \pdv{}x(u_x\tau_{xx} + u_y\tau_{xy} + u_z\tau_{xz}) + \pdv{}y(u_x\tau_{yx} + u_y\tau_{yy} + u_z\tau_{yz}) \cdots\\
    &\qquad+ \pdv{}z(u_x\tau_{zx} + u_y\tau_{zy} + u_z\tau_{zz}) + \pdv{q_x}x + \pdv{q_y}y + \pdv{q_z}z.
  \end{split}
  \label{eq: dpdt}
\end{align}
Most of the terms of \cref{eq: dpdt} are quadratic in $p$ and the velocity components. The exception are the diffusive heat flux terms $\pdv{q_x}x$, $\pdv{q_y}y$, and $\pdv{q_z}z$. Under the constant $c_p$ assumption, these terms are simplified from their original definition:
\begin{align}
  \begin{split}
    q_x  &= -\frac\kappa R \pdv {p\zeta}x + \frac\kappa{c_p} \sum_{l=1}^{n_\text{sp}}\pdv{Y_l}x, \\
    q_y  &= -\frac\kappa R \pdv {p\zeta}y + \frac\kappa{c_p} \sum_{l=1}^{n_\text{sp}}\pdv{Y_l}y, \\
    q_z  &= -\frac\kappa R \pdv {p\zeta}z + \frac\kappa{c_p} \sum_{l=1}^{n_\text{sp}}\pdv{Y_l}z
  \end{split}
\end{align}
which makes \cref{eq: dpdt} quadratic in $p$, $\zeta$, $u_x$, $u_y$, $u_z$, and the species variables. Since the species variables are not directly modeled as part of the lifted state, these heat flux terms are not quadratic in the lifted state.

\subsection{Derivation of $\rho Z_m$ equation}
The evolution of the variable $\rho Z_m$ is given by the conservative representation as
\begin{align}
  \pdv{\rho Z_m}t &=  -\pdv{\rho u_x Z_m}x -\pdv{\rho u_y Z_m}y - \pdv{\rho u_z Z_m}z + \pdv{}x\left(\rho D \pdv{Z_m}x\right)\cdots\nonumber \\
  & \qquad\qquad\qquad\qquad\qquad\qquad + \pdv{}y\left(\rho D \pdv{Z_m}y\right) + \pdv{}y\left(\rho D \pdv{Z_m}y\right).
  \label{eq: rho Z conser}
\end{align}
Note that $\rho D =\rho \alpha= \frac{\kappa}{c_p}$ (under our unit Lewis number assumption), so under the constant $c_p$ assumption, \cref{eq: rho Z conser} becomes
\begin{align}
  \pdv{\rho Z_m}t &=  -\pdv{u_x (\rho  Z_m)}x -\pdv{ u_y (\rho Z_m)}y - \pdv{u_z(\rho  Z_m)}z + \frac{\kappa}{c_p}\left(\pdv[2]{Z_m}x + \pdv[2]{Z_m}y + \pdv[2]{Z_m}z\right).
  \label{eq: dZ eqn final}
\end{align}
The first three terms are quadratic in the velocity components and the variable $\rho Z_m$. For the last three terms, note that the product rule gives $\pdv {\rho Z_m}x = \rho \pdv{Z_m}x + Z_m\pdv\rho x$, which gives the following identity for the $x$-derivative of $Z_m$ (and similar identities exist for $y$ and $z$):
\begin{align}
  \pdv {Z_m}x = \zeta \pdv{\rho Z_m}x - Z_m \zeta \pdv \rho x = \zeta \pdv{\rho Z_m}x + \rho Z_m \pdv{\zeta}x.
  \label{eq: dZdz identity}
\end{align}
Since the first spatial derivatives of $Z_m$ are quadratic in the variable $\rho Z_m$ and the specific volume $\zeta$, the second spatial derivatives must also be quadratic in these variables, since the derivative operator is linear.

\subsection{Derivation of $\rho C$ equation}
The derivation of the $\rho C$-evolution equation under the assumption of constant $c_p$ is analogous to that of the $\rho Z_m$-equation. The evolution equation in the conservative representation is given by
\begin{align}
  \pdv{\rho C}t = -\pdv{\rho u_x C}x - &\pdv{\rho u_y C}y - \pdv{\rho u_z C}z + \pdv{}x\left(\rho D \pdv Cx\right) \cdots\nonumber\\
  & \qquad\qquad + \pdv{}y\left(\rho D\pdv Cy\right) + \pdv{}z\left(\rho D\pdv Cz\right) + \dot\omega_C.
\end{align}
As in the derivation of \cref{eq: dZ eqn final}, we group the variable $\rho C$ together in the first three terms and we use the fact that $\rho D = \frac{\kappa}{c_p}$ is a constant to arrive at
\begin{align}
  \pdv{\rho C}t = -\pdv{u_x(\rho C)}x - &\pdv{u_y (\rho C)}y - \pdv{u_z(\rho  C)}z + \frac\kappa{c_p}\left(\pdv[2]Cx+ \pdv[2]Cy + \pdv[2]Cz\right) +\dot\omega_C. \label{eq: final dCdt}
\end{align}
We have a similar expression for the spatial derivatives of $C$ as we did for those of $Z_m$ in~\cref{eq: dZdz identity}:
\begin{align}
  \pdv Cx = \zeta \pdv Cx - C \zeta \pdv \rho x = \zeta \pdv{\rho C}x + \rho C \pdv{\zeta}x,
\end{align}
which is quadratic in $\zeta$ and $\rho C$. Thus, \cref{eq: final dCdt} is almost fully quadratic in the velocity components, the variable $\rho C$, and $\zeta$, with the sole non-quadratic term being the reaction source term $\dot\omega_C$.

\bibliographystyle{plain}
\bibliography{references}
\end{document}